\documentclass[11pt]{amsart}
\usepackage{amssymb,amsmath}

\makeatletter
\theoremstyle{plain}
 \newtheorem{thm}{Theorem}[section]
\newtheorem{thm*}{Theorem}
 \newtheorem{lem}[thm]{Lemma}
 \newtheorem{prop}[thm]{Proposition}
 \newtheorem{cor}[thm]{Corollary}
 \numberwithin{equation}{section} %% Comment out for sequentially-numbered
\numberwithin{figure}{section} %% Comment out for sequentially-numbered
 \theoremstyle{plain}
 \theoremstyle{definition}
 \newtheorem{defn}[thm]{Definition}
 \newtheorem{rem}[thm]{Remark}
\newcommand{\A}{{{\mathbb A}}}

\newcommand{\p}{{{\mathcal P}}}
\newcommand{\fp}{{{\mathfrak p}}}

\newcommand{\calT}{{{\mathcal T}}}
\newcommand{\calA}{{{\mathcal A}}}
\newcommand{\calS}{{{\mathcal S}}}
\newcommand{\calD}{{{\mathcal D}}}
\newcommand{\calV}{{{\mathcal V}}}
\newcommand{\calM}{{{\mathcal M}}}
\newcommand{\calN}{{{\mathcal N}}}

\newcommand{\calH}{{{\mathcal H}}}

\newcommand{\C}{{{\mathbb C}}}
\newcommand{\R}{{{\mathbb R}}}

\newcommand{\bH}{{{\bf H}}}
\newcommand{\bp}{{{\bf p}}}

\newcommand{\fF}{{{\mathfrak F}}}

\newcommand{\fX}{{{\mathfrak X}}}
\newcommand{\X}{{{\mathbb X}}}

\newcommand{\J}{{{\mathbb J}}}

\newcommand{\fH}{{{\mathfrak H}}}

\newcommand{\binfty}{{{\bf \infty}}}

\makeatother

\date{\today\\
2010 \emph{Mathematics Subject Classifications.} 32Q55, 53B99.\\
\emph{Key words.} Cross--ratios, CR structures, pseudoconformal structures.}
\begin{document}

\title[Psc structures and Falbel's Cross--Ratio Variety]{Pseudoconformal structures and the example of Falbel's Cross--Ratio variety}

\author[I.D. Platis]{Ioannis D. Platis}

\begin{abstract}
We introduce pseudoconformal structures on 4--dimensional manifolds and study their properties. Such structures are arising from two different complex operators which agree in a 2--dimensional subbundle of the tangent bundle; this subbundle thus forms a codimension 2 ${\rm CR}$ structure. A special case is that of a strictly pseudoconformal structure: in this case, the two complex operators are also opposite in a 2-dimensional subbundle which is complementary to the ${\rm CR}$  structure. A non trivial example of a manifold endowed with a (strictly) pseudoconformal structure is Falbel's cross--ratio variety $\fX$; this variety is isomorphic to the ${\rm PU}(2,1)$ configuration space of quadruples of pairwise distinct points in $S^3$. We first prove that there are two complex structures that appear naturally in $\fX$; these give $\fX$ a pseudoconformal structure which coincides with its well known ${\rm CR}$ structure. Using a non trivial involution of $\fX$ we then prove that $\fX$ is a strictly pseudoconformal manifold. The geometric meaning of this involution as well as its interconnections with the $\rm{CR}$ and complex structures of $\fX$ are also studied here in detail.
\end{abstract}

\address{Department of Mathematics and Applied Mathematics, University Campus, Univerity of Crete,  GR-70013 Voutes, Heraklion Crete, Greece}
\email{jplatis@math.uoc.gr}

%\keywords{Cross--ratios, CR structures, pseudoconformal structures}

\maketitle

\section{Introduction and Statement of Results}

In this article we are aiming to reveal the ties between the several structures of Falbel's cross--ratio variety, a set which is isomorphic to the ${\rm PU}(2,1)$ configuration space of four pairwise distinct points in the sphere $S^3$. All these structures are encoded within the property of pseudoconformality, a notion which we also introduce in this work and we shall explain below. 
Our motivation comes from the classical case: to any quadruple $\fp=(p_1,p_2,p_3,p_4)$ of four pairwise distinct points in the Riemann sphere $S^2=\C\cup\{\infty\}$, there is associated their (complex) {\it cross--ratio} which  is the complex number $\X(\fp)$ defined by
$$
\X(\fp)=[p_1,p_2,p_3,p_4]=\frac{p_4-p_2}{p_4-p_1}\cdot\frac{p_3-p_1}{p_3-p_2},
$$
with the obvious modifications if one of the points is $\infty$. There are 24 cross--ratios associated to each such quadruple $\fp$, but due to symmetries it turns out that all possible cross--ratios depend complex analytically on $\X(\fp)$. Letting the group of M\"obius transformations ${\rm PSL}(2,\C)$ of the Riemann sphere act diagonally on the set of quadruples of $S^2$, it is a classical result that the quotient set, that is, the ${\rm PSL}(2,\C)$ configuration space of four points in $S^2$, is isomorphic to $\C\setminus\{0,1\}$ via
$
[\fp]\mapsto \X(\fp).
$
In this manner, it follows that the configuration space admits the structure of a 1--dimensional complex manifold, which is inherited from $\C\setminus\{0,1\}$.

A far more complicated situation appears in the case of $\fF$, the {\it ${\rm PU}(2,1)$ configuration space of quadruples in $S^3$}. $\fF$ is the space of quadruples of pairwise distinct points in $S^3$, factored by the diagonal action of the projective unitary group ${\rm PU}(2,1)$. The sphere $S^3$ is identified via stereographic projection to $\fH\cup\{\infty\}$, where $\fH$ is the Heisenberg group. Recall that $\fH$ is the 2--step nilpotent Lie group with underlying manifold $\C\times \R$ and multiplication law
$$
(z,t)\star(w,s)=\left(z+w, t+s+2\Im(z\overline{w})\right),
$$
for each $(z,t),(w,s)\in\C\times\R$. The set $\fH\cup\{\infty\}$ is also the boundary of complex hyperbolic plane $\bH^2_\C$ and the projective unitary group ${\rm PU}(2,1)$ is the group of holomorphic isometries of $\bH^2_\C$; it acts doubly transitively on $\partial\bH^2_\C=S^3$. There is a metric defined in $\fH$, called the Kor\'anyi metric; the group ${\rm PU}(2,1)$ is generated by the similarities of this metric together with an inversion, see Section \ref{sec:boundary} for details. Kor\'anyi and Reimann defined in \cite{KR1} a complex cross-ratio $\X(\fp)$ associated to a quadruple $\fp$ of pairwise distinct points $p_i=(z_i,t_i)$, $i=1,\dots,4$ in $S^3=\fH\cup\{\infty\}$ in the following manner: If $p\in S^3$, let $\calA(z,t)=|z|^2-it$
if $p=(z,t)\in\fH$ and $\calA(p)=\infty$ if $p=\infty$. Then the {\it complex cross--ratio} $\X(\fp)$ is
$$
\X(\fp)=\frac{\calA\left(p_4\star p_2^{-1}\right)}{\calA\left(p_3\star p_1^{-1}\right)}\cdot\frac{\calA\left(p_4\star p_1^{-1}\right)}{\calA\left(p_3\star p_2^{-1}\right)},
$$
with the obvious modifications if one of the points is $\infty$. This cross--ratio is invariant under the diagonal action of ${\rm PU}(2,1)$ in $\partial\bH^2_\C$. Falbel showed in \cite{F}, that the 24 cross--ratios associated to a quadruple $\fp$ satisfy symmetries analogous to those in the classical case. However, in this case all possible cross--ratios corresponding to a given quadruple depend real analytically on the following three:
$$
\X_1=\X_1(\fp)=[p_1,p_2,p_3,p_4],\quad \X_2=\X_2(\fp)=[p_1,p_3,p_2,p_4]\quad\text{and}\quad\X_3=\X_3(\fp)=[p_2,p_3,p_1,p_4],
$$
which  satisfy the next two conditions:
\begin{eqnarray*}
 &&
 |\X_2|=|\X_3||\X_1|,\\
 &&
 2|\X_1|^2\cdot\Re(\X_3)=|\X_1|^2+|\X_2|^2-2\Re(\X_1+\X_2)+1.
\end{eqnarray*}
These two equations define a 4--dimensional real subvariety $\fX$ of $\C^3$, which we call the {\it (Falbel's) cross--ratio variety}.
It has been shown in \cite{FP} that $\mathfrak{F}$ is isomorphic to $\fX$ via the isomorphism
$$
\varpi: \mathfrak{F}\ni [\fp]\mapsto \left(\X_1(\fp),\X_2(\fp),\X_3(\fp)\right)\in\fX,
$$
see also Section \ref{sec:X-XR} for further details. There is a special involution $\calT$ of $\fX$ which appears naturally on $\fX$. This is given by
$$
\calT(\X_1,\X_2,\X_3)=\left(\X_1,\X_2,\overline{\X_3}\right).
$$
This involution was first observed by the authors of \cite{PP} and its geometric action was characterised there as very mysterious; if $\X_i=\X_i([\fp])$, $\fp=(p_1,p_2,p_3,p_4)$, $i=1,2,3$, then $(\varpi^{-1}\circ\calT\circ\varpi)[\fp]$ is not induced by any permutation of the $p_i$ (with the exception of the case where all $p_i$ lie on a $\C-$circle; then $(\varpi^{-1}\circ\calT\circ\varpi)[\fp]=[\fp]$). In this paper $\calT$ plays a crucial role; this role shall be gradually unfolded in the sequel. 

Besides Falbel's own results, see \cite{F}, cross--ratio variety $\fX$ has been studied in \cite{PP}, \cite{PP2}, see also \cite{CG} for a different approach. An extensive study of the geometric structures of $\fX$ is \cite{FP}; here are the main results: 

First, there is 4--dimensional real manifold structure on a subset $\fX'$ of $\fX$ (see Theorem \ref{thm:X'-sub} below). The  inverse image $\varpi^{-1}(\fX')$ comprises equivalent classes of a quadruples $(p_1,p_2,p_3,p_4)$ such that  not all $p_i$ lie in a $\C-$circle (for the definition of a $\C-$circle, see Section \ref{sec:boundary}). 

Secondly, there is a ${\rm CR}$ structure $\calH$ of codimension 2 defined on a subset $\fX''$ of $\fX$ (see Theorem \ref{thm:X''-CR} below; details about $\rm{CR}$ structures are in Section \ref{sec:cr}). The  inverse image $\varpi^{-1}(\fX'')$ comprises equivalence classes of  quadruples $(p_1,p_2,p_3,p_4)$ such that $p_1,p_2,p_3$ do not all lie in a $\C-$circle. We show in Theorem \ref{thm:CRanti} that this structure is actually an {\it antiholomorphic} ${\rm CR}$ submanifold structure: the distribution $\calV=\calT_*(\calH)$ ($\calT_*$ is the derivative of $\calT$) is complementary to the distribution $\calH$ that defines the ${\rm CR}$ structure of $\fX$; moreover, if $\J$ denotes the natural complex structure of $\C^3$, then $\J\calV\cap T(\fX^*)=\{0\}$. 

Thirdly, there is a structure of a 2--dimensional disconnected complex manifold biholomorphic to $\C P^1\times (\C\setminus\R)$, defined on a subset $\fX^*$ of $\fX$ (see Theorem \ref{thm:X*-J} below); we denote this structure by $J$. The  inverse image $\varpi^{-1}(\fX^*)$ comprises equivalent classes of a quadruples $(p_1,p_2,p_3,p_4)$ such that $p_2,p_3$ do not lie in the same orbit of the stabiliser of $p_1,p_4$.

In this paper we show that apart from the complex structure $J$, there also exists another complex structure defined on $\fX^*$ which we shall denote by $I$. This structure is inherited from a Levi strictly pseudoconvex subset $\p$ of $\C^2$, see Section \ref{sec:complexI}. 

We therefore wish to clarify the relations of all the afore mentioned structures, namely the complex structures $I$ and $J$ and the codimension 2 antiholomorphic $\rm{CR}$ structure $(\calH,\calV)$ of the cross--ratio variety which is induced by the involution $\calT$. It will eventually turn out that in the subset $\fX^*$ of $\fX$, which is a complex manifold with respect both the complex structures $I$ and $J$, we have the following:
\begin{enumerate}
 \item $T(\fX^*)=\calH\oplus\calV$, (here $(\calH,\calV)$ is the restriction of the antiholomorphic $\rm{CR}$ structure to $\fX^*$);
 \item Both $I$ and $J$ are bundle automorphisms of both $\calH$ and $\calV$ and moreover, $I\equiv J$ is on $\calH$ and $I\equiv-J$ on $\calV$.
\end{enumerate}

4-dimensional manifolds $M$ with an antiholomorphic $\rm{CR}$ structure $(\calH,\calV)$ and with complex structures $I$ and $J$ which satisfy condition (2) above, have an interest of their own. Therefore we choose to obtain our result on the cross--ratio variety via a more general setting: We introduce the notions of {\it pseudoconformal (psc)} and {\it strictly pseudoconformal (apsc)} manifolds, which we study in detail in Section \ref{sec:CRpsc}. 
We remark at this point that the term {\it pseudoconformal} which inspired our definition, is found for instance  in p. 138 of \cite{B}; there it is used as an alternative for ${\rm CR}$ {\it mappings}.\footnote{B. McKay pointed out to me that it was actually E. Cartan who first used the term, see \cite{Ca}.}  

A pseudoconformal manifold is actually a $\rm{CR}$ manifold whose $\rm{CR}$ structure arises from two (a priori different) complex structures defined on the manifold. To be more precise, let $M$ be a 4--dimensional real manifold and suppose that it is endowed with two complex structures $I$ and $J$.
Suppose also that there exist a 1--complex dimensional subbundle $\calH^{(1,0)}(M,I)$ of the $(1,0)-$tangent bundle $T^{(1,0)}(M,I)$, such that 
$$
(id._*)\calH^{(1,0)}(M,I)=\calH^{(1,0)}(M,J),
$$
where $\calH^{(1,0)}(M,J)$ is a 1--complex dimensional subbundle of $T^{(1,0)}(M,J)$ and $(id._*)$ is the differential of the identity mapping $id.:(M,I)\to(M,J)$. To avoid trivial cases (i.e., that $I$ and $J$ are either the same or opposite on $M$, in other words the cases where the identity mapping $id.:(M,I)\to(M,J)$ is either holomorphic or antiholomorphic), we require from the subbundle $\calH^{(1,0)}(M,I)$ to be the maximal one with this property: Maximality here is meant in terms of both dimension and uniqueness, that is, if $\calH'\subset T^{(1,0)}$ is such that $(id._*)\calH'\subset T^{(1,0)}(M,J)$, then $\calH'=\calH^{(1,0)}(M,I)$.
We then call the triple $(M,I,J)$ a {\it psc manifold}. 

Let $\calH(M)$ %and $\calH'(M)$
be the underlying real subbundle of  $\calH^{(1,0)}(M,I)$ and  $\calH^{(1,0)}(M,J)$, which we call the {\it horizontal bundle}. Consider $M$ as a real manifold with the complex operators $I$ and $J$ acting as bundle automorphisms on $\calH(M)$. % and $\calH'(M)$, respectively.
Then $(\calH(M),I)$ and $(\calH(M),J)$ are ${\rm CR}$ structures of codimension 2 in $M$, and the main observation is that by the very definition of the psc manifold we have that %:
%$$
%(id.)_*(\calH(M),I)=(\calH(M),J).
%$$
%Thus 
the map $id.:M\to M$ is also a ${\rm CR}$ diffeomorphism. Therefore the two ${\rm CR}$ structures of codimension 2 in $M$ are equivalent. It also follows that psc manifolds have the property that {\it the complex structures $I$ and $J$ are identified on the underlying real bundle $\calH(M)$}. Clearly, a psc  structure on $M$ induces a unique ${\rm CR}$ structure on $M$ by identifying the two equivalent ${\rm CR}$ structures. But the converse does not hold in general; for this to happen, the bundle automorphisms $I$ and $J$ have to be extended to integrable almost complex automorphisms of $M$. The problem of finding natural cases at which this can be done lead us to the next class of manifolds.

Let $(M,I,J)$ be a psc manifold which also has the following property: There exist splittings of  $T^{(1,0)}(M,I)$, $T^{(1,0)}(M,J)$ into direct sums
$
%T^{(1,0)}(M,I)=
\calH^{(1,0)}(M,I)\oplus\calV^{(1,0)}(M,I)$ and $%T^{(1,0)}(M,I)=
\calH^{(1,0)}(M,I)\oplus\calV^{(1,0)}(M,I),
$ respectively,
such that $$(id._*)\calV^{(1,0)}(M,I)=\overline{\calV^{(1,0)}(M,J)}=\calV^{(0,1)}(M,J).$$ Then $(M,I,J)$ shall be called a {\it strictly pseudoconformal (spsc) manifold}. 
Let  $\calV(M)$ %and $\calV'(M)$
be the underlying real subbundle of  $\calV^{(1,0)}(M,I)$ and  $\calV^{(1,0)}(M,J)$, which we call the {\it vertical bundle}; % , respectively. 
%The pairs  $(\calV(M),I)$ and $\calV(M),J)$ may be also considered  as ${\rm CR}$ structures of codimension 2 in $M$ (in a broad sense, where maximality is not required), but now  we have:
we have that
$$
(id.)_*(\calV(M),I)=(\calV(M),-J).
$$
%Thus the identity map is anti--${\rm CR}$; therefore the two ${\rm CR}$ structures of codimension 2 are anti--equivalent. 
All in all, spsc manifolds enjoy the property we wish to prove that $\fX^*$ also enjoys: Their real 4--dimensional tangent bundle admits a decomposition
$$
T(M)=\calH(M)\oplus\calV(M),
$$
where $\calH(M)$ and $\calV(M)$ are the underlying 2--dimensional real subbundles of the horizontal and the vertical bundles of $M$ respectively, and are such that $I\equiv J$ on $\calH(M)$ and
 $I\equiv -J$ on $\calV(M)$.

Note that in contrast with the case of a psc manifold where there is no information about the relation of the complex structures $I$ and $J$ away from the horizontal bundle, in the case of a spsc manifold the relation of $I$ and $J$ is determined precisely on the whole of the tangent space. In principle, one cannot obtain a spsc structure out of an arbitrary psc manifold $(M,I,J)$. However, this actually happens in the case where there is an involution $\calT_M$ of $M$ which fixes no point of $M$ and is holomorphic with respect to one of the complex structures and antiholomorphic with respect to the other, see Corollary \ref{cor:psc-spsc}. 
We finally mention that they may exist {\it singular sets} in a psc (resp. apsc) structure $(M,I,J)$: 
These comprise points $p\in M$ such that
\begin{enumerate}
\item $\calH_p(M)=\{0\}$ %$(id.)_{*,p}\calH(M)=\{0\}$
in the psc case and 
\item $\calH_p(M)=\{0\}$ and $\calV_p(M)=\{0\}$ %$(id.)_{*,p}\calH(M)=\{0\}$ or $(id.)_{*,p}\calV(M)=\{0\}$ 
in the apsc case.
\end{enumerate}

\medskip

The subset $\fX^{*}$ of the cross--ratio variety $\fX$ constitutes a concrete, non trivial example of a psc as well as of a spsc manifold. %as well as of a apsc manifold. 
In fact we have (see Section \ref{sec:main1}):

%\medskip

\begin{thm}\label{thm:X*-psc}
Away from a certain singular set, the triple $(\fX^*,I,J)$ is a psc manifold. Moreover, the ${\rm CR}$ submanifold structure of $(\fX^*,I,J)$ induced by its pseudoconformality coincides with the ${\rm CR}$ submanifold structure, as this is defined  in Section \ref{sec:CR}.
\end{thm}

Using the crucial fact that the involution $\calT:\fX\to\fX$ is $I-$holomorphic and $J-$antiholomorphic we obtain as a corollary the following:

 \begin{thm}\label{thm:X*-spsc}
  Away from a certain singular set, the psc manifold $(\fX^*,I,J)$ is spsc. Moreover, the ${\rm CR}$ submanifold structure of $(\fX^*,I,J)$ induced by its strict pseudoconformality coincides with the antiholomorphic ${\rm CR}$ submanifold structure, as this is defined  in Section \ref{sec:CR}.
\end{thm}

However, the question of revealing the mysterious geometric nature of $\calT$ remains; in Section \ref{sec:tg} we prove a theorem (see Theorem \ref{thm:giT}) which actually says that if $\fp$ is a quadruple of pairwise distinct points which do not all lie in a $\C-$circle and $[\fp']=(\varpi^{-1}\circ\calT\circ\varpi)([\fp])$, then the points of $\fp'$ are obtained out of the points of $\fp$ after applying to the latter suitable elements of ${\rm PU}(2,1)$. These elements are congruent to Heisenberg similarities which depend only on the cross--ratios $\X_i$, $i=1,2,3$ of $\fp$. %explains the relation of the points of two quadruples $\fp$ and $\fp'$ of pairwise distinct points of the boundary with cross--ratios $\X_i$ and $\X_i'$, $i=1,2,3$, respectively, which are such that $\calT(\X_1,\X_2,\X_3)=(\X_1',\X_2',\X_3')$.

\medskip

This paper is organised as follows. In Section \ref{sec:CRpsc} we review some well known facts about ${\rm CR}$ structures and we introduce psc and spsc structures in 4--dimensional manifolds. Section \ref{sec:X-variety} is a broad review of the well known manifold, ${\rm CR}$ and complex structures of $\fX$. For clarity, and due to the different conventions about $\fX$ considered in \cite{FP}, we repeat the proofs of these results here. The proof of Theorems \ref{thm:X*-psc} and \ref{thm:X*-spsc} lies in Section \ref{sec:Xpsc}. Section \ref{sec:tg} is devoted to the geometric interpretation of the involution $\calT$. Finally, in Section \ref{sec:app} we provide furter details about the cross--ratio variety, especially about its several singular sets.

\section{Pseudoconformal Structures}\label{sec:CRpsc}
The  material in this section is divided in two parts. The content of the first part is well known; there is a vast bibliography about ${\rm CR}$ structures which we review in Section \ref{sec:cr}, see for instance \cite{B}, \cite{C}, \cite{DT}. In Section \ref{sec:psc}
 we study the notions of what we shall call pseudoconformal mappings and pseudoconformal manifolds. Both these notions are very much alike  ${\rm CR}$ mappings and ${\rm CR}$ manifolds; the case we are interested in is that of codimension 2, but there are direct generalisations. 

\subsection{Preliminaries: ${\rm CR}$ structures}\label{sec:cr}

There are two equivalent definitions of an {\it abstract ${\rm CR}$ structure}. Suppose first that
$M$ is a $(2p+s)-$dimensional real manifold. A {\it  ${\rm CR}$ structure of codimension $s$ in $M$} is a pair $(\calD, J)$ where $\calD$ is a $2p-$dimensional smooth subbundle of $T(M)$ and $J$ is a bundle automorphism of $\calD$ such that:
\begin{enumerate}
 \item [{(i)}] $J^2=-id.$ and
 \item [{(ii)}] if $X$ and $Y$ are sections of $\calD$ then the same holds for $[X,Y]-[JX,JY]$, $[JX,Y]+[X,JY]$ and moreover
 $
 J\left([X,Y]-[JX,JY]\right)=[JX,Y]+[X,JY].
 $
\end{enumerate}
On the other hand, 
let $M$ be a $(2p+s)-$dimensional real manifold and let $T^\C(M)$ be its complexified tangent bundle. A ${\rm CR}$ structure of codimension $s$ in $M$ is a complex $p-$dimensional smooth subbundle $\calH$ of $T^\C(M)$ such that:
\begin{enumerate}
 \item [{(i)}] $\calH\cap\overline{\calH}=\{0\}$ and
 \item [{(ii)}] $\calH$ is involutive, that is for any vector fields $Z$ and $W$ in $\calH$ we have
 $
 [Z,W]\in\calH.
 $
\end{enumerate}
%\end{defn}
The two definitions are equivalent; see for instance Theorem 1.1, Chpt. VI of \cite{B}. A manifold endowed with a ${\rm CR}$ structure is called a {\it ${\rm CR}$ manifold}. A special class of ${\rm CR}$ manifolds are generic %the  ${\rm CR}$  
submanifolds of complex manifolds:
Suppose that $N$ is a complex manifold of complex dimension $n$ with complex structure $J$, and let $M$ be a submanifold of $N$ of real dimension $m$. Then if  %the map $\calH:x\to\calH_x$, where
$$
\calH=T(M)\cap J(T(M)), %\quad x\in M
$$
i.e., the maximal invariant subspace of $T(M)$ under the action of $J$, is also a smooth subbundle on $M$, then $M$ is called a {\it generic submanifold of $(N,J)$}.
A generic submanifold is in fact a ${\rm CR}$ manifold (see for instance Theorem 2.1, p.135 of \cite{B}). The ${\rm CR}$ structure is $(\calH, J)$, where here by $J$ we denote the bundle automorphism induced by the restriction of $J$ in $\calH$. The corresponding complex subbundle is
$$
\calH^{(1,0)}=\{Z\in T^\C(M)\;|\;Z=X-iJX,\;X\in\calH\},
$$
and we have
$$
X\in\calH\quad\text{if and only if}\quad Z=X-iJX\in\calH^{(1,0)}.
$$
Suppose finally that $M$ is a generic submanifold of the $n-$dimensional complex manifold $N$ with $n=p+s$, such that $\dim_\R M=2p+s$, where $2p=\dim_\R\calH$. %; that is, $M$ is a codimension $s$ ${\rm CR}$ submanifold of $N$. 
Let $\calV$ be a subbundle of $M$ complementary to $\calH$:
$$
T(M)=\calH\oplus\calV.%\quad x\in M.
$$
Note that $\dim_\R\calV=s$. %Moreover, we always have
%$$
%J(\calV_x)\cap T_x(M)\subset J(T_x(M))\cap T_x(M)=\calH_x
%$$
%for each $x\in M$.
If 
$$
J(\calV)\cap T(M)=\{0\}, %\quad x\in M,
$$
we call $M$ an {\it antiholomorphic ${\rm CR}$ submanifold} of $N$, see p. 136 of \cite{B}. %Note that  if $\calV^\C=\calV\otimes\C$ is the complexification of $\calV$, %  $$\calV^{(1,0)}=\{V\in T^\C(M)\;|\;V=U-iJU,\;U\in\calV\},$$ 
%then the above condition reads equivalently as
%$$
%\calH^{(1,0)}\cap\calV^\C=\{0\}.
%$$

${\rm CR}$ diffeomorphisms are defined as follows: Let $M$ and $M'$ be ${\rm CR}$ manifolds of the same dimension $m=2p+s$ with ${\rm CR}$ structures $\calH$ and $\calH'$ respectively, of the same dimension $s$. A diffeomorphism $F:M\to M'$ is a ${\rm CR}$ {\it diffeomorphism} if it preserves ${\rm CR}$ structures; that is
 $
 F_*\calH=\calH'.
 $
In other words, $F$ is a ${\rm CR}$ diffeomorphism if and only if for each $Z\in\calH$ we have $F_*Z\in\calH '$. In terms of the corresponding real distributions $(\calD,J)$ and $(\calD',J')$ we may say that $F$ is ${\rm CR}$ if for each $X\in\calD$ we have
$
F_*(JX)=J'(F_*X).
$

In this paper we are concerned in particular with
%\subsubsection{
${\rm CR}$ structures on subvarieties of $\C^n$. %}\label{sec:CR-sub}
We consider the manifold $\C^n$, $n>1$, with the natural complex coordinates $(\zeta_1,\dots,\zeta_n)$, $\zeta_i=x_i+iy_i$, $i=1,\dots, n$.  Denote also by $\J$ the natural complex structure of $\C^n$. An $m-$real dimensional smooth subvariety  of $\C^n$  is locally defined by a set of equations
\begin{eqnarray*}
 F_i(\zeta_1,\dots,\zeta_n)=0,\quad i=1,\dots, k=2n-m.
\end{eqnarray*}
The set $M$ consisting of points of the subvariety at which the matrix
$$
D=\left[\begin{matrix}
       \frac{\partial F_1}{\partial x_1}&\dots&\frac{\partial F_1}{\partial x_n}&\frac{\partial F_1}{\partial y_1}&\dots&\frac{\partial F_1}{\partial y_n}\\
       \\
       \vdots&\dots&\vdots&\vdots&\dots&\vdots\\
       \\
       \frac{\partial F_k}{\partial x_1}&\dots&\frac{\partial F_k}{\partial x_n}&\frac{\partial F_k}{\partial y_1}&\dots&\frac{\partial F_k}{\partial y_n}\\
      \end{matrix}\right]
$$
is of constant rank $k$ is a real submanifold of $\C^n$ with $\dim(M)=m$. Its tangent space $T_x(M)$ at a point $x\in M$ is identified to the set
$$
T_x(M)=\{X\in T_x(\C^n)\;|\; (dF_i)_x(X)=0,\; i=1,\dots,k\}.
$$
The maximal complex subspace $\calH_x$ at each $x\in M$ comprises of $X\in T_x(\C^n)$ such that
$$
(dF_i)_x(X)=0\;\;\text{and}\;\;(d^cF_i)_x(X)=0,\quad i=1,\dots k,\quad\text{where}\;\;(d^cF_i)_x(X)=-(dF_i)_x(\J X).
$$
Let %$\calH^{(1,0)}_x$ 
$$
\calH^{(1,0)}_x=\{Z=X-i\J X\in T^{(1,0)}(\C^n)\;|\;X\in\calH_x\}. 
$$
Then $\calH^{(1,0)}_x$ comprises of $Z\in T_x^{(1,0)}(\C^n)$ such that
$$
(\partial F_i)_x(Z)=0, \quad  i=1,\dots k,
$$
and one verifies  that
$$
X\in\calH_x\;\;\text{if and only if}\;\; Z=X-i\J X\in \calH^{(1,0)}_x.
$$
Denote by $\calH^{(1,0)}$ the complex subbundle comprising of $\calH^{(1,0)}_x$, $x\in M$. At points $x\in M$ consider the matrix
$$
D^{(1,0)}=\left[\begin{matrix}
       \frac{\partial F_1}{\partial \zeta_1}&\dots&\frac{\partial F_1}{\partial \zeta_n}\\
       \\
       \vdots&\dots&\vdots\\
       \\
       \frac{\partial F_k}{\partial \zeta_1}&\dots&\frac{\partial F_k}{\partial \zeta_n}\\
      \end{matrix}\right],
$$
and let $M'\subset M$ be the set at which $D^{(1,0)}$ is of constant rank $l\le k$. Then  $\calH^{(1,0)}$ is defined at $M'$, $\dim_\C\calH^{(1,0)}=n-l=p$ and therefore, if the integrability condition
$$
[Z_i,Z_j]\in \calH^{(1,0)}\quad\text{for every}\;\; i,j=1,\dots p,\; i\neq j,
$$
holds, then  $\calH^{(1,0)}$ is a ${\rm CR}$ structure of codimension $s=2l-k$ since $m=2n-k=2p+s$.
We call the set $\calS=M\setminus M'$ the {\it singular set} of the ${\rm CR}$ structure.

In the particular case when $k=l=n-1$, that is $\dim_\C\calH^{(1,0)}=1$, the single vector field generating the ${\rm CR}$ structure is
\begin{equation*}
Z=D_{\zeta_2,\dots,\zeta_n}\frac{\partial}{\partial \zeta_1}+D_{\zeta_3,\dots,\zeta_n,\zeta_1}\frac{\partial}{\partial \zeta_2}+\dots +D_{\zeta_1,\dots,\zeta_{n-1}}\frac{\partial}{\partial \zeta_n}, 
\end{equation*} 
where
$$
D_{\zeta_{i_1},\zeta_{i_{n-1}}}=\left|\frac{\partial(F_1,\dots, F_{n-1})}{\partial(\zeta_{i_1},\dots,\zeta_{i_{n-1}})}\right|
$$
are the $(n-1)-$minor subdeterminants of $D^{(1,0)}$.
Note that in this case, the above integrability condition holds vacuously. Also in this case, the {\it Levi form}
$
(L)_p:\calH^{(1,0)}_p\to\R^n
$
is defined in $M'$  by
$$
Z_p\mapsto (L_1(p),\dots,L_{n-1}(p))=\left(dd^cF_1(Z,\overline{Z})_p,\dots,dd^cF_{n-1}(Z,\overline{Z})_p\right),
$$
where
$$
L_i(p)=\left[\begin{matrix}
                                     D_{\zeta_2,\dots,\zeta_n}&\dots&D_{\zeta_1,\dots,\zeta_{n-1}}
                                    \end{matrix}\right]_p\cdot
\left[\begin{matrix}
\frac{\partial^2 F_i}{\partial \zeta_1\partial\overline{\zeta_1}}&\dots& \frac{\partial^2 F_i}{\partial \zeta_1\partial\overline{\zeta_n}}\\
\vdots&\dots&\vdots\\
\frac{\partial^2 F_i}{\partial \zeta_n\partial\overline{\zeta_1}}&\dots& \frac{\partial^2 F}{\partial \zeta_n\partial\overline{\zeta_n}}
\end{matrix}\right]_p\cdot
\left[\begin{matrix}D_{\overline{\zeta_2},\dots,\overline{\zeta_n}}\\
\vdots\\
D_{\overline{\zeta_1},\dots,\overline{\zeta_{n-1}}}
                                    \end{matrix}\right]_p.
$$

\subsection{Pseudoconformal Structures}\label{sec:psc}

Pseudoconformal structures that we are about to define are quite relative to ${\rm CR}$ structures.  The main difference is that whether in the ${\rm CR}$ case there is no prescribed complex structure, in the pseudoconformal case there are prescribed two a priori different complex structures. 

\subsubsection{Pseudoconformal mappings and submanifolds}

We start with a definition.

\begin{defn}\label{defn:psc}
 Let $(M,I)$ and $(N,J)$ be complex manifolds with complex structures $I$ and $J$ respectively, $\dim_\C(M)=2$ and $\dim_\C(N)=n\ge 2$.
 A smooth immersion $F:M\to N$ shall be called  {\it pseudoconformal (psc)} if there exists a 1--complex dimensional subbundle 
$
\calH^{(1,0)}(M,I)$ of $T^{(1,0)}(M,I)
$ such that $$F_*\calH^{(1,0)}(M,I)\subset T^{(1,0)}(N,J),
$$
and $\calH^{(1,0)}(M,I)$ is the maximal subbundle with this property: If $(\calH')^{(1,0)}(M,I)$ is any subbundle of $T^{(1,0)}(M,I)
$ such that $F_*(\calH')^{(1,0)}(M,I)\subset T^{(1,0)}(N,J)
$, then $\calH^{(1,0)}(M,I)= (\calH')^{(1,0)}(M,I)
$. We call $\calH^{(1,0)}(M,I)$ (resp. $\calH^{(1,0)}(F(M),J)$) the {\it horizontal bundle} of $M$ (resp. of $N$). In the case where $n>2$ and $M$ is an immersed submanifold of $N$, % (i.e., the inclusion $\iota:M\hookrightarrow N$ is an embedding), 
 then the manifold $(M,I)$ is called a {\it psc submanifold (of codimension 2)} of $(N,J)$.
\end{defn}
When there is no risk of confusion, the underlying 2--dimensional real subbundles $\calH(M)$ and $\calH(F(M))$ of $\calH^{(1,0)}(M,I)$ and  $\calH^{(1,0)}(F(M),J)$ respectively, shall be also called horizontal bundles (of $M$ and $N$ respectively).

At this point we wish to make some remarks on Definition \ref{defn:psc}. The first remark has to do with the restriction $\dim_\C\calH^{(1,0)}(N,J)=1$; if it is replaced by $\dim_\C\calH^{(1,0)}(N,J)=2$, then by the Newlander--Nirenberg Theorem $F$ is holomorphic ($F_*I=JF_*$ everywhere on $T(M)$), and $N$ is a complex submanifold of $M$.  By putting the maximality condition in Definition \ref{defn:psc} we do not allow such a case. Moreover, we cannot either have that $F$ is antiholomorphic, $F_*I=-JF_*$ everywhere on $T(M)$; this contradicts the equality $F_*I=JF_*$ on $\calH(M)$.
Second, in the case where $m=2$ and $F$ is a smooth diffeomorphism, it is clear that  $F$ is psc if and only if $F^{-1}$ is psc; such diffeomorphisms are studied below. But perhaps the most important observation is given by the next proposition whose proof follows directly from Definition \ref{defn:psc}.  

\begin{prop}
 Let $(M,I)$ and $(N,J)$ be as above and let $F$ be a psc immersion with horizontal bundle $\calH^{(1,0)}(M,I)$. Denote by $\calH(M)$  the underlying 2--real dimensional real subbundle of $\calH^{(1,0)}(M,I)$. Then the following hold:
 \begin{enumerate}
  \item Consider $M$ and $N$ as real manifolds and their complex structures as bundle automorphisms such that $I^2=J^2=-id.$, acting only on  $\calH(M)$ and $\calH(F(M))$, respectively. Then $M$ and $N$ are ${\rm CR}$ manifolds of codimension 2 and the ${\rm CR}$ structures are $(\calH(M),I)$ and $(F_*\calH(M),J)$ respectively.  
  \item $F:M\to N$ is a ${\rm CR}$ map; the immersed submanifold $F(M)$ is a generic submanifold of $M$ with a codimension 2 ${\rm CR}$ structure.
 \end{enumerate} 
\end{prop}

Hence psc mappings are ${\rm CR}$ mappings;  the converse is of course not generally true. The following proposition gives a useful local description of psc mappings.

\begin{prop}\label{prop:psc}
The smooth immersion $F:(M,I)\to(N,J)$ is psc if and only if at each point $p\in M$, there exists a local parametrisation $(\zeta_1,\zeta_2)\mapsto(\xi_1,\dots,\xi_n)$ of $F$, ($(\zeta_1,\zeta_2)$ are local $I-$holomorphic coordinates around $p$ and $(\xi_1,\dots,\xi_m)$ are local $J-$holomorphic coordinates around $F(p)$),  such that
\begin{equation}\label{eq:psc}
{\rm rank} (DF^{(0,1)})=1,\quad\text{where}\quad DF^{(0,1)}=\left[\begin{matrix}
\frac{\partial\overline{\xi_1}}{\partial \zeta_1}&\frac{\partial\overline{\xi_1}}{\partial \zeta_2}\\
\vdots &\vdots\\
\frac{\partial\overline{\xi_m}}{\partial \zeta_1}&\frac{\partial\overline{\xi_m}}{\partial \zeta_2}
                                                  \end{matrix}\right].
\end{equation}
\end{prop} 

\begin{proof}
We first prove that $F$ is psc if and only if  condition (\ref{eq:psc}) holds for {\it each} local representation $(\zeta_1,\zeta_2)\mapsto(\xi_1,\dots,\xi_m)$ of $F$. 
For this, let  $Z\in T^{(1,0)}(M,I)$ with a local representation $Z=\sum_{i=1}^2a_i(\partial/\partial \zeta_i)$. Then
\begin{eqnarray*}
 F_*Z&=&\sum_{i=1}^2a_i\;F_*\left(\frac{\partial}{\partial \zeta_i}\right)\\
&=&\sum_{i=1}^2a_i\;\sum_{j=1}^m\frac{\partial \xi_j}{\partial \zeta_i}\left(\frac{\partial}{\partial \xi_j}\right)+
\sum_{i=1}^2a_i\;\sum_{j=1}^m\frac{\partial \overline{\xi_j}}{\partial \zeta_i}\left(\frac{\partial}{\partial \overline{\xi_j}}\right).
\end{eqnarray*}
 Therefore, and due to linear independence, $F_*Z\in T^{(1,0)}(N,J)$ if and only if the linear system
\begin{eqnarray*}
&&
\frac{\partial \overline{\xi_1}}{\partial \zeta_1}\;a_1+\frac{\partial \overline{\xi_1}}{\partial \zeta_2}\;a_2=0,\\
&&
\quad \quad \quad \vdots \\
&&
\frac{\partial \overline{\xi_m}}{\partial \zeta_1}\;a_1+\frac{\partial \overline{\xi_m}}{\partial \zeta_2}\;a_2=0,
\end{eqnarray*}
admits non zero solutions, which is equivalent to condition \ref{eq:psc}.

For the converse, fix $p\in M$ and let $(\zeta_1,\zeta_2)\mapsto(\xi_1,\dots,\xi_m)$ be the local representation of $F$ around $p$ in which (\ref{eq:psc}) holds. If $(\zeta_1',\zeta_2')\mapsto(\xi_1',\dots,\xi_n')$ is another local representation around $p$ with corresponding matrix $D\tilde F^{(0,1)}$, consider the holomorphic change of coordinates $(\zeta_1,\zeta_2)\mapsto (\zeta_1',\zeta_2')$ and $(\xi_1,\dots,\xi_m)\mapsto(\xi_1',\dots,\xi_m')$. Then, from the chain rule we have
\begin{equation*}
D\tilde F^{(0,1)}=
\frac{\partial(\overline{\xi_1'},\dots,\overline{\xi_m'})}{\partial(\overline{\xi_1},\dots,\overline{\xi_m})}\cdot
DF^{(0,1)}\cdot
\frac{\partial(\zeta_1,\zeta_2)}{\partial(\zeta_1',\zeta_2')},
\end{equation*} 
and our assertion is proved.
\end{proof}

\medskip

We comment here that ${\rm rank} (DF^{(0,1)})=1$ is equivalent to saying that all minor $2\times 2$ subdeterminants $\left|\frac{\partial(\overline{\xi_j},\overline{\xi_k})}{\partial(\zeta_1,\zeta_2)}\right|$ vanish for each $j,k=1,\dots, m$, $j\neq k$ and the partial derivatives do not vanish simultaneously. However, points of $M$ where all partial derivatives vanish simultaneously might exist; the set of these points is defined below.

\begin{defn}\label{defn:sing-psc}
Let $F:(M,I)\to (N,J)$ be a psc mapping. The set
$$
\calS=\left\{p\in M\;|\;(F_{*,p})\calH_p^{(1,0)}(M,I)=0\right\},
$$
is called the {\it singular set} of $F$.
\end{defn}
\subsubsection{Antiholomorphic  pseudoconformal  submanifolds}

The following proposition connects psc immersions of a certain nature with antiholomorphic ${\rm CR}$ submanifold structures.
\begin{prop}\label{prop:psc-anti}
 Suppose that $(M,I)$ is a 2--dimensional manifold which is pseudoconformally immersed into the $n-$complex dimensional complex manifold $(N,J)$, $n>2$. Let $\iota:M\hookrightarrow N $ be the inclusion and let $\calH(M)$ be  the underlying real subbundle of the horizontal bundle $\calH^{(1,0)}(M,I)$. Suppose in addition that there exists an 2--real dimensional subbundle $\calV(M)$ of the tangent bundle $T(M)$ such that:
 $$
T(M)=\calH(M)\oplus\calV(M),
$$
Then $M$ inherits from $\iota$  an antiholomorphic ${\rm CR}$ submanifold structure of codimension 2.
\end{prop}
\begin{proof}
  Let $\widetilde{T}(M)$ be the tangent bundle of $M$ in $N$ and let also  $\widetilde{\calH}(M)=\iota_*\calH(M)$ and  $\widetilde{\calV}(M)=$ $\iota_*\calV(M)$. We have
$$
\widetilde{T}(M)= \widetilde{\calH}(M)\oplus\widetilde{\calV}(M),
$$
therefore
\begin{eqnarray*}
 J\widetilde{T}(M)&=& J\widetilde{\calH}(M)\oplus J\widetilde{\calV}(M)\\
 &=&\iota_* I\calH(M)\oplus J\widetilde{\calV}(M)\\
 &=&\iota_*\calH(M)\oplus J\widetilde{\calV}(M)\\
 &=&\widetilde{\calH}(M)\oplus J\widetilde{\calV}(M)
\end{eqnarray*}
and thus
\begin{eqnarray*}
 \widetilde{\calH}(M)&=&\widetilde{T}(M)\cap J\widetilde{T}(M)\\
 &=&\left(\widetilde{\calH}(M)\oplus\widetilde{\calV}(M)\right)\oplus\left(\widetilde{\calH}(M)\oplus J\widetilde{\calV}(M)\right)\\
 &=&\widetilde{\calH}(M)\oplus\left(\widetilde{\calV}(M)\cap J\widetilde{\calV}(M) \right).
\end{eqnarray*}
Hence we must have $\widetilde{\calV}(M)\cap J\widetilde{\calV}(M)=\{0\}$ and we conclude that
$$
J\widetilde{\calV}(M)\cap\widetilde{T}(M)=J\widetilde{\calV}(M)\cap\left(\widetilde{\calH}(M)\oplus\widetilde{\calV}(M)\right)=\{0\}.
$$
Therefore $M$ is an antiholomorphic ${\rm CR}$ submanifold of codimension 2 of $N$.
\end{proof}
Since $(M,I)$ is a complex manifold we may always assume that $\calV(M)$ is $I-$invariant. In this case, we also have a splitting
$$
T^{(1,0)}(M,I)=\calH^{(1,0)}(M,I)\oplus \calV^{(1,0)}(M,I),
$$
where $\calV^{(1,0)}(M,I)=\{W-iIW\;|\;W\in\calV(M)\}$. We call $\calV(M)$ (and  $\calV^{(1,0)}(M,I)$) the {\it vertical bundle of} $M$.
\begin{rem}
Let $\widetilde{\calV}^\C=\iota_*\calV^{(1,0)}(M,I)$. Then
\begin{equation}\label{eq:compV}
\widetilde{\calV}^\C\cap\widetilde{T}^{(1,0)}(M,J)=\{0\}\quad\text{and}\quad \widetilde{\calV}^\C\cap\widetilde{T}^{(0,1)}(M,J)=\{0\}.
\end{equation}
Indeed, left relation of (\ref{eq:compV}) holds because of pseudoconformality. To show the right relation, suppose that there exists a $W\in\calV(M)$ and an $X\in\widetilde{T}(M)$ such that
$$
\iota_*W-i\iota_*(IW)=X+iJX.
$$ 
Therefore  $X=\widetilde{W}=\iota_*W\in\widetilde{V}(M)$ and also $\iota_*(IW)=-JX\in J\widetilde{\calV}(M)$. Applying $J$ to the last relation we have
$X=J(\iota_*(IW))\in J\widetilde{\calV}(M)$. From the proof of Proposition \ref{prop:psc-anti} we deduce $X=0$.

We conclude that $\widetilde{\calV}^\C$ consists of complex vector fields of mixed type with respect to $J$, i.e., if $\widetilde{W}\in\widetilde{\calV}^\C$ then $\widetilde{W}=\widetilde{W}_1+\widetilde{W}_2$ where $\widetilde{W}_1$ is of type $(1,0)$, $\widetilde{W}_2$ is of type $(0,1)$ and neither of which is zero.
\end{rem}

An antiholomorphic psc submanifold $(M,I)$ of $(N,J)$ is trivially a psc submanifold of $(N,J)$. The converse is also true in the following case which is of our particular interest.

\begin{cor}\label{cor:psc-anti}
Suppose that $(M,I)$ is a psc submanifold of $(N,J)$, $\dim_\C(M)=2$. Suppose that there exist a holomorphic involution $\calT_M$ of $M$. %and an antiholomorphic involution $\calT_N$ of $N$ such that if $\iota:M\hookrightarrow N$ is the inclusion, then
%$$
%\calT_N\circ\iota=\iota\circ\calT_M.
%$$
Then $(M,I)$ is an antiholomorphic psc submanifold of $(N,J)$. 
\end{cor}

\begin{proof}
Let $\calH^{(1,0)}(M,I)=\langle Z\rangle$, $Z\in T^{(1,0)}(M,I)$, $Z=X-iIX$. Then $W=(\calT_M)_*(Z)\in T^{(1,0)}(M,I)$ since $\calT_M$ is $I-$holomorphic. Let $\calH(M)=\langle X,IX\rangle$ and define 
$$
\calV(M)=\langle (\calT_M)_*(X),(\calT_M)_*(IX)=I\left((\calT_M)_*(X)\right) \rangle.
$$ 
Since $\calT_M$ is an involution, $\calH(M)\cap\calV(M)=\emptyset$ and since also  $\dim_\C(M)=2$, $T(M)=\calH(M)\oplus\calV(M)$. The result now follows from Proposition \ref{prop:psc-anti}.
\end{proof}

\subsubsection{Pseudoconformal diffeomorphisms and  manifolds}\label{sec:psc-diffs}
Suppose that $(N,J)$ and $(M,I)$ are as before but now $\dim_\C(N)=\dim_\C(M)=2$ and $F:(M,I)\to(N,J)$ is a psc diffeomorphism. We call  $(M,I)$ and $(N,J)$ {\it pseudoconformally equivalent} and the following corollary is evident.

\begin{cor}\label{cor:psc}
A smooth diffeomorphism $F:(M,I)\to(N,J)$ is psc if and only if at each point $p\in M$ there exists a local representation $(\zeta_1,\zeta_2)\mapsto(z_1,z_2)$ of $F$,( $(\zeta_1,\zeta_2)$ are local $I-$holomorphic coordinates around  $p$ and $(z_1,z_2)$ are local $J-$holomorphic coordinates around $F(p)$), such that
\begin{equation}\label{eq:psc-diff}
\left|\frac{\partial(\overline{z_1},\overline{z_2})}{\partial(\zeta_1,\zeta_2)}\right|=0\;\text{(which is equivalent to}\; \left|\frac{\partial(\overline{\zeta_1},\overline{\zeta_2})}{\partial(z_1,z_2)}\right|=0),
\end{equation}
where the partial derivatives involved do not vanish simultaneously.
\end{cor}

We are mostly interested in the case when  $M=N$ is a 4--dimensional real manifold endowed with complex structures $I$ and $J$ arising from atlantes $\calA_I$ and $\calA_J$ respectively. If these structures are psc equivalent, i.e., the identity mapping $id.:(M,I)\to(M,J)$ is psc, then we call $(M,I,J)$ a {\it psc manifold}.  From the previous discussion it follows that the  horizontal bundles $\calH^{(1,0)}(M,I)$ and $\calH^{(1,0)}(N,J)$ are holomorphically identified via the derivative $id._*$, that is, $I\equiv J$ on the underlying real bundle $\calH(M)$. The resulting bundle via this identification, which we shall denote simply by $\calH^{(1,0)}(M)$, is a 1--complex dimensional subbundle of $T(M)\otimes\C$; it is also a ${\rm CR}$ structure of codimension 2 in $M$ in the usual sense when the action of the complex structures are considered only in $\calH^{(1,0)}(M)$. 

\medskip

The next proposition follows directly from Corollary \ref{cor:psc} and shows that in a psc manifold $(M,I,J)$ the unified atlas $\calA=\calA_I\cup\calA_J$  is such that the  transition maps from charts of $\calA_I$ to charts of $\calA_J$  are psc diffeomorphisms of open sets in $\C^2$.

\begin{prop}\label{prop:psc-man}
Let $(M,I,J)$ be a 4-dimensional real manifold   with complex structures $J$ and $I$ arising from the atlantes $\calA_J=\{(U_j,\phi_j)\}$ and $\calA_I=\{(V_i,\psi_i)\}$ respectively. Then $(M,I,J)$ is psc  if and only if for each $p\in M$ there exist $(U_p,\phi_p)\in\calA_J$ and $(V_p,\psi_p)\in\calA_I$ so that the map $\psi\circ\phi^{-1}:\phi_p(U_p\cap V_p)\to\psi_p(U_p\cap V_p)$ is a psc diffeomorphism. Explicitly, if  $\phi_p:q\mapsto (\zeta_1,\zeta_2)$ for each $q\in U_p$ and $\psi_p:r\mapsto (z_1,z_2)$ for each $r\in V_p$ and 
 $\psi\circ\phi^{-1}$ is given by $(\zeta_1,\zeta_2)\mapsto(z_1,z_2)$, then Condition (\ref{eq:psc-diff})  holds.
\end{prop}

We also have the following proposition which is quite expected: Equivalent psc structures immersed in $\C^n$ produce equivalent ${\rm CR}$ submanifold structures.

\begin{prop}\label{prop:psc-CR}
Let $M$ and $N$ be 2--dimensional pseudoconformally equivalent complex manifolds with complex structures $J$ and $I$ respectively. 
Let also $\calH^{(1,0)}(M,I)$ and $\calH^{(1,0)}(N,J)$ be the horizontal bundles of $(M,I)$ and $(N,J)$ respectively.
We   suppose that there exists a psc immersion $F_M:(M,I)\to\C^n$ so that $$(F_M)_*\calH^{(1,0)}(M,I)=\calH^{(1,0)}$$ is a ${\rm CR}$ structure of codimension 2 on $M$.

Then the map $F_N:(N,J)\to\C^n$ defined by $F_N=F_M\circ F^{-1}$ is psc and $(F_N)_*\calH^{(1,0)}(N,J)=\calH^{(1,0)}$.
\end{prop}

\begin{proof}
 From psc equivalence of $(M,I)$ and $(N,J)$ we have that there exists  a smooth psc diffeomorphism $F:(M,I)\to(N,J)$ so that $$F_*\calH^{(1,0)}(M,I)=\calH^{(1,0)}(N,J),$$ Since $F_M$ is an immersion of $M$ into $\C^3$ and $F$ is a diffeomorphism from $M$ onto $N$, we have that $F_N=F_M\circ F^{-1}$ is an immersion of $N$ into $\C^3$.  Now $F_N$ is a psc embedding, since
 \begin{equation*}
 (F_N)_*\calH^{(1,0)}(N,J)=\left((F_M)_*\circ F_*^{-1}\right)F_*\calH^{(1,0)}(M,I)=(F_M)_*\calH^{(1,0)}(M,I)=\calH^{(1,0)}.
 \end{equation*}
 This completes the proof.
\end{proof}

\begin{cor}\label{cor:psc-CR}
 Let $(M,I,J)$ be a 2--complex dimensional psc manifold with horizontal bundle $\calH^{(1,0)}(M)$.   If there exists a psc embedding $\iota:(M,I)\to\C^n$ so that $\iota_*\calH^{(1,0)}(M)$ constitutes a ${\rm CR}$ submanifold structure of codimension 2 of $M$, then there exists a psc  embedding  $j:(M,J)\to\C^n$ so that $j_*\calH^{(1,0)}=\iota_*\calH^{(1,0)}$. Therefore, the ${\rm CR}$ submanifold structure of $M$ may be identified to both ${\rm CR}$ structures arising from the psc structure of $(M,I,J)$.
\end{cor}

\subsubsection{Strictly pseudoconformal diffeomorphisms and manifolds}\label{sec:strict-psc}

Among the class of psc diffeomorphisms we distinguish one which is defined as follows.
\begin{defn}\label{defn:strictly-psc}
Let $(M,I)$ and $(N,J)$ be 2--dimensional complex manifolds which are psc equivalent via the psc diffeomorphism  $F:(M,I)\to(N,J)$ and let $\calH^{(1,0)}(M,I)$ be the horizontal bundle of $(M,I)$. %Besides the horizontal bundle $\calH^{(1,0)}(M,I)$ the subbundle of $T^{(1,0)}(M,I)$ which is such that $F_*\calH^{(1,0)}(M,I)\in T^{(1,0)}(N,J)$, 
We assume
additionally the existence of  a 1--dimensional complex  subbundle $\calV^{(1,0)}(M,I)$ of $T^{(1,0)}(M,I)$ such that 
\begin{enumerate}
 \item [{i)}] $\calH^{(1,0)}(M,I)\oplus\calV^{(1,0)}(M,I)=T^{(1,0)}(M,I)$ and
\item [{ii)}] the $F_*-$image of $\calV^{(1,0)}(M,I)$ is an 1--dimensional complex subbundle of \\ $T^{(0,1)}(N,J)$.
\end{enumerate}
Such an $F$ shall be called {\it strictly pseudoconformal (spsc)}, the subbundle $\calV^{(1,0)}(M,I)$ shall be called {\it vertical bundle} and the manifolds $(M,I)$ and $(N,J)$ shall be called {\it strictly pseudoconformally equivalent}.
\end{defn}

\medskip

Working analogously as in the previous paragraph, we can prove the counterparts of Proposition \ref{prop:psc} and Corollary \ref{cor:psc}. 

\begin{prop}\label{prop:strictly-psc}
Let $(M,I)$ and $(N,J)$ be 2--dimensional complex manifolds. The smooth diffeomorphism $F:(M,I)\to(N,J)$ is strictly psc if and only if at each point $p\in M$ there exists a local representation $(z_1,z_2)\mapsto(\zeta_1,\zeta_2)$ of $F$, ($(z_1,z_2)$ are local $J-$holomorphic coordinates around $p$ and $(\zeta_1,\zeta_2)$ are local $I-$holomorphic coordinates around $F(p)$), such that
\begin{equation}\label{eq:strictly-psc}
{\rm rank} (DF^{(1,0)})={\rm rank} (DF^{(0,1)})=1,
\quad DF^{(1,0)}=\left[\begin{matrix}
\frac{\partial\zeta_1}{\partial z_1}&\frac{\partial\zeta_2}{\partial z_1}\\
\\
\frac{\partial\zeta_2}{\partial z_1}&\frac{\partial\zeta_2}{\partial z_1}
                                                  \end{matrix}\right],\quad DF^{(0,1)}=\left[\begin{matrix}
\frac{\partial\overline{\zeta_1}}{\partial z_1}&\frac{\partial\overline{\zeta_2}}{\partial z_1}\\
\\
\frac{\partial\overline{\zeta_2}}{\partial z_1}&\frac{\partial\overline{\zeta_2}}{\partial z_1}
                                                  \end{matrix}\right].
\end{equation}
Equivalently,
\begin{equation}\label{eq:strictly-psc-diff}
\left|\frac{\partial(\zeta_1,\zeta_2)}{\partial(z_1,z_2)}\right|=\left|\frac{\partial(\overline{\zeta_1},\overline{\zeta_2})}{\partial(z_1,z_2)}\right|=0,
\end{equation}
where the partial derivatives involved do not vanish simultaneously. 
\end{prop} 
It is clear that Equation (\ref{eq:strictly-psc-diff}) is equivalent to
\begin{equation*}
 \left|\frac{\partial(z_1,z_2)}{\partial(\zeta_1,\zeta_2)}\right|=\left|\frac{\partial(\overline{z_1},\overline{z_2})}{\partial(\zeta_1,\zeta_2)}\right|=0.
\end{equation*}

\medskip

It is natural to ask at this point under which conditions a psc diffeomorphism is also a spsc diffeomorphism. To this direction we have the following proposition.

\begin{prop}\label{prop:psc-spscdiff}
Let $F:(M,I)\to (N,J)$ be a psc diffeomorphism. Suppose that there exist a holomorphic involution $\calT_M$ of $M$ and an antiholomorphic involution $\calT_N$ of $M$ such that $\calT_N\circ F=F\circ\calT_M$. Then $F$ is spsc.
\end{prop}
\begin{proof}
 If $\calH^{(1,0)}(M,I)=\langle Z\rangle$ let $W=(\calT_M)_*(Z)$. Since $\calT_M$ is holomorphic, $W\in T^{(1,0)}(M,I)$. Let $\calV^{(1,0)}(M,I)=\langle W\rangle$; since $\calT_M$ is an involution,
 $$
 \calH^{(1,0)}(M,I)\cap\calV^{(1,0)}(M,I)=\emptyset
 $$
 and $\{Z,W\}$ is a basis for $T^{(1,0)}(M,I)$. Moreover,
 $$
 F_*(W)=(F_*\circ(\calT_M)_*)(Z)=((\calT_N)_*\circ F_*)(Z)=(\calT_N)_*(F_*(Z))\in T^{(0,1)}(N,J).
 $$
\end{proof}

We have mentioned in the previous section that in a psc manifold $(M,I,J)$ we have an identification of complex structures $I$ and $J$ on the horizontal bundle $\calH(M)$; in general, there is no information about the relation of $I$ and $J$ elsewhere on the tangent space of $M$. On the class of psc manifolds we are about to define, this relation is transparent.

A {\it strictly psc manifold} $(M,I,J)$ is a psc manifold with the property that the identity mapping $id.:(M,I)$ $\to(N,J)$ is strictly psc. In this case, besides the holomorphic identification of horizontal bundles there is an antiholomorphic identification of vertical bundles $\calV^{(1,0)}(M,I)$ and $\calV^{(0,1)}=id._*\calV^{(1,0)}$ respectively. The resulting underlying real bundle shall be denoted by $\calV(M)$. It is clear from the definition that in a spsc manifold, its complex structures $I$ and $J$ have the following properties:
\begin{enumerate}
 \item $I=J$ on $\calH(M)$;
 \item $I=-J$ on $\calV(M)$.
\end{enumerate}

\begin{prop}\label{prop:spsc-man}
Let $(M,I,J)$ be a psc manifold  with complex structures $J$ and $I$ arising from the atlantes $\calA_J=\{(U_j,\phi_j)\}$ and $\calA_I=\{(V_i,\psi_i)\}$ respectively. Then $(M,I,J)$ is strictly psc  if and only if for each $p\in M$ there exist $(U_p,\phi_p)\in\calA_J$ and $(V_p,\psi_p)\in\calA_I$ so that the map $\psi\circ\phi^{-1}:\phi_p(U_p\cap V_p)\to\psi_p(U_p\cap V_p)$ is a strictly psc diffeomorphism. Explicitly, if  $\phi_p:q\mapsto (\zeta_1,\zeta_2)$ for each $q\in U_p$ and $\psi_p:r\mapsto (z_1,z_2)$ for each $r\in V_p$ and 
 $\psi\circ\phi^{-1}$ is given by $(\zeta_1,\zeta_2)\mapsto(z_1,z_2)$ then Condition (\ref{eq:strictly-psc-diff})  holds.
\end{prop}

The existence of an involution of a certain kind in a psc manifold $(M,I,J)$ ensures that it is also a spsc manifold.

\begin{cor}\label{cor:psc-spsc}
Let $(M,I,J)$ be a psc manifold and suppose that there exist an involution $\calT$ of $M$ which is $I-$holomorphic and $J-$antiholomorphic. Then $(M,I,J)$ is a spsc manifold.
\end{cor}
\begin{proof}
 The proof follows directly from Proposition \ref{prop:psc-spscdiff}.
\end{proof}

The next proposition is the spsc counterpart of Proposition \ref{prop:psc-CR}.

\begin{prop}\label{prop:spsc-CR}
Let $M$ and $N$ be 2--dimensional complex manifolds with complex structures $J$ and $I$ respectively,  which are strictly pseudoconformally equivalent via the spsc map $F:(M,I)$ $\to(N,J)$. 
We also  suppose that there exists an antiholomorphic  psc immersion $F_M:(M,I)\to\C^n$ so that 
$$(F_M)_*\calH^{(1,0)}(M,I)=\calH^{(1,0)},\quad (F_M)_*\calV^{(1,0)}(M,I)=\calV^\C,
$$
and the underlying real subbundles $\calH$ and $\calV$ of $\calH^{(1,0)}$ and $\calV^\C$ respectively, form an antiholomorphic ${\rm CR}$ submanifold structure of codimension 2 of $M$.

Then the map $F_N:(N,J)\to\C^n$ defined by $F_N=F_M\circ F^{-1}$ is an antiholomorphic psc immersion  and $(F_N)_*\calH^{(1,0)}(N,J)=\calH^{(1,0)}$, $(F_N)_*\calV^{(0,1)}(N,J)=\overline{\calV^\C}$.
\end{prop}

\begin{proof}
From spsc equivalence, there exists  a smooth spsc diffeomorphism $F:(M,I)\to(N,J)$ so that $$F_*\calH^{(1,0)}(M,I)=\calH^{(1,0)}(N,J)\quad\text{and}\quad F_*\calV^{(1,0)}(M,I)=\calV^{(0,1)}(N,J), $$
where $\calH^{(1,0)}(M,I)$, $\calH^{(1,0)}(N,J)$ are the horizontal bundles and $\calV^{(1,0)}(M,I)$, $\calV^{(0,1)}(N,J)$ are the vertical bundles of $(M,I)$ and $(N,J)$, respectively.
 We only have to prove the last equation and we do so by proving the equivalent relation:
 $$
 (F_N)_*\calV^{(0,1)}(N,J)=\calV^\C.
 $$
 We indeed have
 \begin{equation*}
 (F_N)_*\calV^{(0,1)}(N,J)=\left((F_M)_*\circ F_*^{-1}\right)F_*\calV^{(1,0)}(M,I)=(F_M)_*\calV^{(1,0)}(M,I)=\calV^\C
 \end{equation*}
 and the proof is complete.
\end{proof}

\medskip

\begin{cor}\label{cor:spsc-CR}
 Let $(M,I,J)$ be a 2--complex dimensional strictly psc manifold.  If there exists an antiholomorphic  psc embedding $\iota:(M,I)\to\C^n$ giving $M$ the structure of an antiholomorphic ${\rm CR}$ submanifold of codimension 2,   then there exists an antiholomorphic  psc  embedding  $j:(M,J)\to\C^n$ which gives $M$ the same antiholomorphic ${\rm CR}$ submanifold structure. %Therefore, the antiholomorphic ${\rm CR}$ submanifold structure of $M$ may be ientified to both ${\rm CR}$ structures arising from the psc structure of $(M,I,J)$.
\end{cor}

%We close this section by showing how a spsc manifold can be constructed out of a generic CR submanifold which enjoys a certain property. 
%\begin{prop}\label{prop:spsc-invo}
%Let $M$ be a generic CR submanifold of the complex manifold $(N,J)$, $\dim_\R(M)$ $=4$, $\dim_\C(N)=3$, with a codimension 2 CR structure $(\calH,J)$. Assume that there exists an involution $\calT_N$ of $N$ such that:
%\begin{enumerate}
% \item $\calT_N$ is neither holomorphic nor antiholomorphic;
% \item $\calT_N(M)=M$ and $T_N$ does nowhere leave $M$ pointwise fixed.
%\end{enumerate}

%\end{prop}
%\begin{proof}
% Let
%$
%\calH^{(1,0)}=\langle Z\rangle
%$
%and let $W=(\calT_M)_*(Z)$. Since  $\calT_N$ is neither holomorphic nor antiholomorphic we have
%$$
%W=W^{(1,0)}+W^{(0,1)},
%$$
%where none of $W^{(1,0)}$ and $W^{(0,1)}$ is zero. 
%\end{proof}

%$M$ of a complex manifold $(N,J)$with a codimension 2 CR structure $\calH(M)$, which is also endowed with an involution $\calT_M$ which satisfies a certain condition.
%Let $\calV(M)=(\calT_M)_*(\calH)$.

\section{Cross--Ratio Variety}\label{sec:X-variety}
This section contains a review of all well known results about Falbel's cross--ratio variety; the result of Theorem \ref{thm:CRanti} is new. This review is quite extended, partly for clarity and partly  due to the different conventions considered for $\fX$ in \cite{F, FP}. As preliminaries to cross--ratio variety we discuss basic facts about complex hyperbolic plane and its boundary in Section \ref{sec:prel}. We define the cross--ratio variety $\fX$ and we discuss its relation with $\fF$, the ${\rm PU}(2,1)-$configuration space of four points in $S^3$ (Section \ref{sec:X-XR}). Singular sets and the involution $\calT$ of $\fX$ are in Sections \ref{sec:sing-sets} and \ref{sec:T}, respectively.   %In Section \ref{sec:alternative} we show the relations between cross--ratios and Cartan's invariants and we also give an alternative description of $\fX$. 
%Significant and singular subsets of $\fX$ are in section \ref{sec:sing-sets}.
 Manifold,  ${\rm CR}$ and complex structures in $\fX$  are  in Section \ref{sec:X-structures}. % and Section \ref{sec:CR} respectively. We introduce a new complex manifold structure in Section \ref{sec:complexI} and we describe the already known complex structure in Section \ref{sec:complexJ}.

\subsection{Preliminaries to Cross--Ratio Variety }\label{sec:prel}
The  material in this section is well known; for details we refer the reader to the standard book of Goldman \cite{Gol}. Complex hyperbolic plane is treated in Section \ref{sec:chs} and its boundary in Section \ref{sec:boundary}. Definitions of Cartan's angular invariant and complex cross--ratio are in Section \ref{sec:invariants}. 

\subsubsection{Complex Hyperbolic Plane}\label{sec:chs}
We consider $\mathbb{C}^{2,1}$,  the vector space $\mathbb{C}^{3}$  with
the Hermitian form of signature $(2,1)$ given by
$$
\left\langle {\bf {z}},{\bf {w}}\right\rangle 
=z_{1}\overline{w}_{3}+z_{2}\overline{w}_{2}+z_{3}\overline{w}_{1},
$$
and consider the following subspaces of ${\mathbb C}^{2,1}$:
\begin{equation*}
V_-  =  \Bigl\{{\bf z}\in{\mathbb C}^{2,1}\ :\ 
\langle{\bf z},\,{\bf z} \rangle<0\Bigr\}, \quad
V_0  =  \Bigl\{{\bf z}\in{\mathbb C}^{2,1}\setminus\{{\bf 0}\}\ :\ 
\langle{\bf z},\,{\bf z} \rangle=0\Bigr\}.
\end{equation*}
Denote by  ${\mathbb P}:{\mathbb C}^{2,1}\setminus\{0\}\longrightarrow {\mathbb C}P^2$ 
the canonical projection onto complex projective space. Then the
{\sl complex hyperbolic plane} ${\bf H}_{\mathbb{C}}^{2}$
is defined to be ${\mathbb P}V_-$ and its boundary
$\partial{\bf H}^2_{\mathbb C}$ is ${\mathbb P}V_0$.
Hence  we have
$$
{\bf H}^2_{\mathbb C} = \left\{ (z_1,\,z_2)\in{\mathbb C}^2
\ : \ 2\Re(z_1)+|z_2|^2<0\right\},
$$
and in this manner, ${\bf H}^2_{\mathbb C}$ is the Siegel domain in 
${\mathbb C}^2$.

There are two distinguished points in $V_0$  which we denote by 
${\bf o}$ and $\binfty$:
$$
{\bf o}=\left[\begin{matrix} 0 \\ 0 \\ 1 \end{matrix}\right], \quad
\binfty=\left[\begin{matrix} 1 \\ 0 \\ 0 \end{matrix}\right].
$$
Let ${\mathbb P}{\bf o}=o$ and ${\mathbb P}\binfty=\infty$. 
Then 
$$
\partial{\bf H}^2_{\mathbb C}\setminus\{\infty\} 
=\left\{ (z_1,\,z_2)\in{\mathbb C}^2
\ : \ 2\Re(z_1)+|z_2|^2=0\right\},
$$
and in particular, $o=(0,0)\in\C^2$. 

Conversely, if we are given a point $z=(z_1,z_2)$ of 
${\mathbb C}^2$, then the point 
$$
{\bf z}=\left[\begin{matrix} z_1 \\ z_2 \\ 1 \end{matrix}\right].
$$
is called the {\sl standard lift} of $z$. Therefore the standard lifts of points of the complex hyperbolic plane and its boundary (except infinity) are  vectors of $V_-$ and $V_0$ respectively with the third inhomogeneous coordinate equal to 1.

$ {\bf H}_{\mathbb{C}}^{2}$ is a K\"ahler manifold, its K\"ahler structure is given by the Bergman metric. 
The holomorphic sectional curvature  
equals to $-1$ and its real sectional curvature
is pinched between $-1$ and $-1/4$.
%\subsubsection{Isometries, complex lines, Lagrangian planes}
%Let ${\rm U(2,1) }$ be the group
%of unitary matrices for  the  Hermitian form 
%$\left\langle \cdot,\cdot\right\rangle $. Each such matrix $A$ satisfies
%the relation $A^{-1}=JA^{*}J$ where $A^{*}$ is the Hermitian transpose of $A$.
The full group of holomorphic isometries %of 
%complex hyperbolic space 
is the \textsl{projective
unitary group}% ${\rm PU(2,1)}={\rm U(2,1)}/{\rm U(1)}$, where 
%${\rm U(1)}=\{ e^{i\theta}I,\theta\in[0,2\pi)\}$
%and $I$ is the $3\times3$ identity matrix.  We have
%We find sometimes  convenient to consider instead the group ${\rm SU(2,1)}$
%of matrices which are unitary with respect to 
%$\left\langle \cdot,\cdot\right\rangle $,
%and have determinant $1$. 
%Therefore 
$${\rm PU(2,1)}={\rm SU(2,1)}/\{ I,\omega I,\omega^{2}I\},$$
where $\omega$ is a non real cube root of unity  (that is ${\rm SU}(2,1)$
is a 3-fold covering of ${\rm PU}(2,1)$). %This is the direct analogue
%of the fact that ${\rm SL}(2,\C)$ is the double cover of
%${\rm PSL}(2,\C)$.
There are two ways (up to ${\rm PU}(2,1)$ conjugacy) to embed real hyperbolic plane into complex hyperbolic plane; that is, as $\bH^1_\C$ as well as $\bH^2_\R$. These embeddings give rise to  complex lines, i.e., isometric images of the embedding of $\bH^1_\C$ into $\bH^2_\C$ and   Lagrangian planes, i.e., isometric images of $\bH^2_\R$ into $\bH^2_\C$, respectively.

\subsubsection{The boundary--Heisenberg group}\label{sec:boundary}
There is an identification of the boundary of the Siegel domain with
the one point compactification of $\C\times\R$: A
 finite point $z$  in the boundary of the Siegel domain has a  standard
lift of the form%to $\C^{2,1}$ is ${\bf z}$ where
%$$
%{\bf z}=\left[\begin{matrix} z_1 \\ z_2 \\ 1 \end{matrix}\right]
%\quad \text{ where }\quad z_1 + \bar{z}_{1} + |z_2|^2 = 0.
%$$
%We write $z=z_2/\sqrt{2}\in\C$ and this 
%condition becomes $2\Re(z_1)=-2|z|^2$. Hence we may
%write $z_1=-|z|^2+it$ for $t\in\R$. That is for $z\in\C$ and
%$t\in\R$:
$$
{\bf z}
=\left[\begin{matrix} -|z|^2+it \\ \sqrt{2}z\\ 1\end{matrix}\right].
$$
The unipotent stabiliser at infinity acts simply transitively and gives the set of these points the structure of a 2--step nilpotent Lie group, namely the Heisenberg group $\fH$. This is $\C\times\R$ with group law:
$$
(z,t)\star (w,s)=(z+w,t+s+2\Im(\overline{w}z)).
$$
The Heisenberg norm (Kor\'anyi gauge) is given by
$$
\left|(z,t)\right|_\fH=\left| \calA(z,t)\right|^{1/2},\quad\text{where}\quad \calA(z,t)=|z|^2-it.
$$
From this norm arises  a metric, the Kor\'anyi--Cygan (K--C) metric, on $\fH$ by the relation
$$
d_\fH\left((z,t),\,(w,s)\right)
=\left|(z,t)^{-1}\star (w,s)\right|_\fH.
$$
%Or, in other words
%$$
%d_K\left((z_1,t_1),\,(z_2,t_2)\right)
%=\left| |z_1-z_2|^2-it_1+it_2-2i\Im(z_1\bar{z}_2)\right|^{1/2}.
%$$
%By taking the standard lift of points on $\partial{\bf H}^2_\C-\{\infty\}$ 
%to $\C^{2,1}$ we can write the K--C metric as:
%$$
%d_K\left((z_1,t_1),\,(z_2,t_2)\right)
%=\left|\left\langle\left[ \begin{matrix}
%-|z_1|^2+it_1 \\ \sqrt{2}z_1 \\ 1 \end{matrix}\right],\,
%\left[ \begin{matrix}
%-|z_2|^2+it_2 \\ \sqrt{2}z_2 \\ 1 \end{matrix}\right]
%\right\rangle\right|^{1/2}.
%$$
The K--C metric is invariant under 
\begin{enumerate}
 \item the left action of $\fH$, $(z,t)\to(w,s)\star (z,t)$;
%\item the Heisenberg translations $(z,t)\mapsto (z,t+s)$, $s\in\R$;
\item the rotations  $(z,t)\mapsto(ze^{i\phi},t)$, $\phi\in\R$.
\end{enumerate}
These form the  group ${\rm Isom}(\fH,d_K)$ of {\it Heisenberg isometries}. %, is thus represented by the group consisting of matrices of the form
%\begin{equation}
% \left[\begin{matrix}
 %       1&-\sqrt{2}\zeta e^{i\phi}&-|\zeta|^2+is\\
%0&e^{i\phi}&\sqrt{2}\zeta\\
%0&0&1
%       \end{matrix}\right].
%\end{equation}
The K--C metric is also scaled up to multiplicative constants by the action of Heisenberg dilations $(z,t)\mapsto$ $(rz,r^2t)$, $r\in\R_*$ and there is also an inversion $R$, defined for each $p=(z,t)\in\fH$, $p\neq o$, by $(z,t)\mapsto\;\left(\frac{z}{-|z|^2+it},-\frac{t}{|-|z|^2+it|^2|}\right)$% if $(z,t)\neq o,\infty$, $R(o)=\infty$, $R(\infty)=o$ 
which satisfies
$$
d_\fH(R(p),R(p'))=\frac{d_\fH(p,p')}{d_\fH(p,o)d_\fH(p',o)}.
$$%MORE HERE% The group
%$$
%{\rm Sim}(\fH,d_K)=\R\times {\rm U}(1)\times\fH
%$$
%acting on $\fH$ is called the group of Heisenberg similarities.

All the above transformations are extended to infinity in the obvious way and the action of ${\rm PU}(2,1)$ on the boundary is given by compositions of these transformations.

%\subsubsection{$\R-$circles and $\C-$circles}

$\R-$circles are boundaries of Lagrangian planes and $\C-$circles are boundaries of complex lines. They come in two flavours, infinite ones (i.e. containing the point at infinity) and finite ones. We refer to \cite{Gol} for more more details about these curves.

\subsubsection{Invariants: Cartan's Angular Invariant and Complex Cross--Ratio }\label{sec:invariants}

Given a triple $(p_1,p_2,p_3)$ of points at the boundary $\partial\bH^2_\C$ the Cartan's angular invariant $\A(p_1,p_2,p_3)$ is defined by
$$
\A(p_1,p_2,p_3)=\arg(-\langle\bp_1,\bp_2\rangle\langle \bp_2,\bp_3\rangle\langle\bp_3,\bp_1\rangle)
$$
where $\bp_i$ are lifts of $p_i$, $i=1,2,3$.
The Cartan's angular invariant lies in $[-\pi/2,\pi/2]$, is independent of the choice of the lifts and remains invariant under the diagonal action of ${\rm PU}(2,1)$. Any other permutation of points produces angular invariants which differ from the above possibly up to sign. The following propositions are in \cite{Gol} to which we also refer the reader for further details:

\medskip

\begin{prop}
Let $(p_1,p_2,p_3)$ be a triple of points lying in $\partial\bH^2_\C$ and let also $\A=\A(p_1,p_2,p_3)$ be their Cartan's angular invariant. Then
\begin{enumerate}
\item All points lie in an $\R-$circle if and only if $\A=0$.
\item All points lie in a $\C-$circle if and only if $\A=\pm\pi/2$.
\end{enumerate}
\end{prop}

\begin{prop}
 Suppose that $p_i$ and $p'_i$, $i=1,2,3$, are points in $\partial{\bH_\C^2}$. If there exists a holomorphic isometry $g$ of $\bH^2_\C$ such that $g(p_i)=p'_i$, $i=1,2,3$, then $\A(p_1,p_2,p_3)=\A(p'_1,p'_2,p'_3)$. Conversely, if $\A(p_1,p_2,p_3)=\A(p'_1,p'_2,p'_3)$, then there exists a holomorphic isometry $g$ of $\bH^2_\C$ such that $g(p_i)=p'_i$, $i=1,2,3$. This isometry is unique unless  $p_i$, $i=1,2,3$,  lie in a $\C-$circle. %Then  if and only if $\A(p_1,p_2,p_3)=\A(p'_1,p'_2,p'_3)$. If all $p_i$ and all $p'_i$ lie in  $\C-$circles, there exists an isometry $g$ of $\bH^2_\C$ such that $g(p_i)=p'_i$, $i=1,2,3$ if and only if either $\A(p_1,p_2,p_3)=\A(p'_1,p'_2,p'_3)$, or $\A(p_1,p_2,p_3)=-\A(p'_1,p'_2,p'_3)$.
\end{prop}

\medskip

%\subsubsection{Cross--ratios}
Given a  quadruple of pairwise distinct points $\fp=(p_1,p_2,p_3,p_4)$ in $\partial\bH_\C^2$,  we define their complex  cross--ratio as follows:  
$$
\X(p_1,p_2,p_3,p_4)=\frac{\langle \bp_3,\bp_1\rangle \langle \bp_4,\bp_2\rangle}{\langle \bp_4,\bp_1\rangle \langle \bp_3,\bp_2\rangle},
$$
where $\bp_i$ are lifts of $p_i$, $i=1,2,3,4$. %Note that this definition agrees with the one given in the introduction. 
The cross--ratio is independent of the choice of lifts and remains invariant under the diagonal action of ${\rm PU}(2,1)$. We stress here that for points in the Heisenberg group, the square root of its absolute value is
\begin{equation*}
 |\X(p_1,p_2,p_3,p_4)|^{1/2}=\frac{d_K(p_4,p_2)\cdot d_K(p_3,p_1)}{d_K(p_4,p_1)\cdot d_K(p_3,p_2)}.
\end{equation*}

\subsection{Cross--ratio variety and the configuration space}\label{sec:X-XR}

Given a quadruple $\fp=(p_1,p_2,p_3,p_4)$ of pairwise distinct points in the boundary $\partial\bH^2_\C$, all possible permutations of points gives us 24 complex cross--ratios corresponding to $\fp$. Due to  symmetries, see \cite{F}, Falbel showed that all cross--ratios corresponding to a quadruple of points depend on three cross--ratios which satisfy two real equations. Indeed, the following proposition holds; for its proof, see for instance \cite{PP}.

\begin{prop}\label{prop:cross-ratio-equalities}
Let $\fp=(p_1,p_2,p_3,p_4)$ be any quadruple of pairwise distinct points in $\partial \bH^2_\C$. Let
$$
\X_1(\fp)=\X(p_1,p_2,p_3,p_4),\quad \X_2(\fp)=\X(p_1,p_3,p_2,p_4),\quad \X_3(\fp)=\X(p_2,p_3,p_1,p_4).
$$
Then
\begin{eqnarray}\label{eq:cross1}
 &&
|\X_3|^2=|\X_2|^2/|\X_1|^2,\\
&&\label{eq:cross2}
2|\X_1|^2\Re(\X_3)=|\X_1|^2+|\X_2|^2-2\Re(\X_1)-2\Re(\X_2)+1.
\end{eqnarray}
\end{prop}
 
\medskip

\begin{defn}\label{defn:cr-variety}
Equations (\ref{eq:cross1}) and  (\ref{eq:cross2}) define a 4--dimensional real subvariety of $\C^3$ which we call the {\it cross--ratio variety} $\fX$.% is the subset of $\C^3$ at which Eqs. \ref{eq:cross1} and \ref{eq:cross2} are satisfied.
\end{defn}

\medskip

We shall now discuss the relation between cross--ratio variety $\fX$ and $\mathfrak{F}$, the space of ${\rm PU}(2,1)$ configurations of four points in $S^3$. The space $\fF$ consists of equivalence classes $[\fp]$, where $\fp$ is a quadruple of pairwise distinct points in $\partial\bH^2_\C$. Two quadruples $\fp=(p_1,p_2,p_3,p_4)$ and $\fp'=(p_1',p_2',p_3',p_4')$ belong to the same equivalence class, if and only if there exists an element $g\in{\rm PU}(2,1)$ such that $g(p_j)=p_j'$ for each $j=1,2,3,4$.
Consider the map $\varpi:\mathfrak{F}\to\fX$ given by
\begin{equation*}%\label{eq:varpi}
[\fp]\mapsto\left(\X_1(\fp),\X_2(\fp),\X_3(\fp)\right).
\end{equation*}
This map is a surjection as the following proposition (Proposition 5.5, \cite{PP}) shows. 

\medskip

\begin{prop}
 Let $x_1,x_2$ and $x_3$ be three complex numbers satisfying
\begin{equation*}
 |x_3|^2=|x_2|^2/|x_1|^2\quad\text{and}\quad 2|x_1|^2\Re(x_3)=|x_1|^2+|x_2|^2-2\Re(x_1)+\Re(x_2)+1.
\end{equation*}
Then there exist a quadruple of points $\fp=(p_1,p_2,p_3,p_4)$, $p_i\in\partial\bH^2_\C$, $i=1,\dots,4$ so that
$$
\X_1(\fp)=x_1,\quad \X_2(\fp)=x_2,\quad\X_3(\fp)=x_3.
$$
\end{prop}

\medskip

By  Proposition 5.10 of \cite{PP}, the map $\varpi$ is  also 1--1 in a large subspace of $\mathfrak{F}$.

\begin{prop}\label{prop:isometric}
 Let $\fp=(p_1,p_2,p_3,p_4)$ and $\fp'=(p_1',p_2',p_3',p_4')$ be two quadruples of pairwise distinct points in $\partial\bH^2_\C$ such that neither all $p_i$ nor all $p_i'$ lie in the same $\C-$circle. Then there exists an element $g\in{\rm PU}(2,1)$ such that $g(p_j)=p_j'$ for each $j=1,2,3,4$ if and only if $\X_i(\fp)=\X_i(\fp')$  for each $i=1,2,3$. 
\end{prop}

\medskip

Therefore, to each point $[\fp]$ of $\mathfrak{F}$ such that not all $p_i$ lie in a $\C-$circle, there is associated a unique point $\varpi(\fp)=(\X_1(\fp),\X_2(\fp),\X_3(\fp))$ of the cross--ratio variety $\fX$.  In the degenerate case where all $p_i\in\fp$ lie in a $\C-$circle, surjection of $\varpi$ still holds, but injection fails as this was shown in \cite{CG}. Following Lemma 5.5 of the corrected version of \cite{F}, the map $\varpi$ is 2--1  from the space $\mathfrak{F}_\R$ of configurations of points lying in a $\C-$circle to a subset $\fX_\R$  of $\fX$ called the 
{\it real singular set} of $\fX$, see (\ref{eq:XR}) below. Besides $\fX_\R$ there are also other singular sets which we are now about to discuss. %is defined in Section \ref{sec:submanifold} and is sudied in detail in Section \ref{sec:sing-sets}.

\subsection{Singular Sets}\label{sec:sing-sets} The structures of the cross--ratio variety $\fX$ we study in this paper are not globally defined. There are singular sets; in this section we state the definitions of these sets and describe their properties in brief. For the proof of those properties as well as for a more detailed discussion on singular sets, see Section \ref{sec:ssagain}. 
We mentioned above the real singular set $\fX_\R$, which is in 2--1 correspondence with the subset $\fF_\R$ of the configuration space consisting of classes of quadruples of points such that all lie in a $\C-$circle. It turns out that
\begin{equation}\label{eq:XR}
\fX_\R=\left\{(\X_1,\X_2,\X_3)\in\fX\;|\; \X_i\in\R,\; \X_1+\X_2=1,\;\frac{1}{\X_2}+\frac{1}{\X_3}=1,\;\X_3+\frac{1}{\X_1}=1\right\}.
\end{equation}
As a manifold, $\X_\R$ is a straight line with two points removed. Next, we have
the {\it ${\rm CR}$ singular set}:
\begin{equation}\label{eq:XCR}
 \fX_{\rm{CR}}=\left\{(\X_1,\X_2,\X_3)\in\fX\;|\; \; \X_1+\X_2=1,\;\frac{1}{\X_2}+\frac{1}{\X_3}=1,\;\X_3+\frac{1}{\X_1}=1\right\}.
\end{equation}
$\fX_{\rm{CR}}$ is the singular set of the codimension 2 ${\rm CR}$ submanifold structure of $\fX$. This set is a 1--dimensional complex manifold biholomorphic to $\C\setminus\{0,1\}$ and it is in 1--1 correspondence with the subset $\fF_{\C\R}$ of the configuration space which consists of equivalence classes of quadruples $\fp=(p_1,p_2,p_3,p_4)$ such that $p_1,p_2,p_3$ lie in a $\C-$circle. Finally, we consider the {\it complex singular set}:
\begin{equation}\label{eq:XC}
 \fX_\C=\left\{(\X_1,\X_2,\X_3)\in\fX\;|\; \Im(\X_3)=0\right\}.
\end{equation}
The complex structures of $\fX$ we study here, see Sections \ref{sec:complexJ} and \ref{sec:complexI}, may be defined away from $\fX_\C$; this fact justifies the name {\it complex singular set}. $\fX_\C$ is in 1--1 correspondence with the subset $\fF_\C$ of $\fF$ consisting of equivalence classes of quadruples $\fp=(p_1,p_2,p_3,p_4)$ such that $p_2,p_3$ lie in the same orbit of the stabiliser of $p_1,p_4$. $\fX_\C$ has a rich structures itself; besides being a subset of dimension one, it can be endowed with the structure of a 3--dimensional submanifold of $\C^2$. Additionally, it has a ${\rm CR}$ structure of codimension 1 which is simply the restriction of the ${\rm CR}$ structure of $\fX$ in $\fX_{\C\R}$. 

Besides the above singular sets, we are going to consider another one which is obtained by a natural involution of $\fX$.

\subsection{The Involution $\calT$}\label{sec:T}
 We introduce the involution $\calT$ of $\fX$; this is given by
 \begin{equation}\label{eq:T}
  \calT(\X_1,\X_2,\X_3)=\left(\X_1,\X_2,\overline{\X_3}\right),\quad (\X_1,\X_2,\X_3)\in\fX.
 \end{equation}
This involution is studied extensively in Section \ref{sec:tg}. Here, we remark that clearly $\calT$ leaves pointwise fixed the singular set $\fX_\C$. % As for the singular set $\X_\C$, it is the disjoint union of the sets $\fX^*_+$ and $\fX^*_-$, where
%\begin{equation*}
 %\fX^*_+=\{(\X_1,\X_2,\X_3)\in\fX'\;|\;\Im(\X_3)>0\}\quad\text{and}\quad \fX^*_-=\{(\X_1,\X_2,\X_3)\in\fX'\;|\;\Im(\X_3)<0\}.
%\end{equation*}
%Involution $\calT$ swaps $\fX^*_+$ and $\fX^*_-$; 
%Also, we will see below that for the complex structures we define for $\fX$, the involution $\calT$ plays the role of the natural conjugation $z\to\overline{z}$ in the trivial case of $\C$ %It will turn out that the same holds for the complex structure $J$ 
%(see Sections \ref{sec:complexJ} and \ref{sec:complexI}). 
Moreover, the antiholomorphic nature of the ${\rm CR}$ structure we define in \ref{sec:CR} arises from $\calT$.

%In the complex structure we are about to define in t (see Section \ref{sec:complexI}), 

For the moment, we  focus on the $\calT-$image of the singular set $\fX_{\rm{CR}}$. This is the set
\begin{equation}\label{eq:XCRc}
 \fX_{\rm{CR}}^{*}=\left\{(\X_1,\X_2,\X_3)\in\fX\;|\; \; \X_1+\X_2=1,\;\frac{1}{\X_2}+\frac{1}{\overline{\X_3}}=1,\;\overline{\X_3}+\frac{1}{\X_1}=1\right\}.
\end{equation}
We show in Section \ref{sec:ssagain} that $ \fX_{\rm{CR}}^{*}$ is isomorphic to the subset of $\fF$ consisting of classes of quadruples $\fp=(p_1,p_2,p_3,p_4)$ such that $p_2,p_3,p_4$ lie in a $\C-$circle. As a manifold, it is isomorphic to $ \fX_{\rm{CR}}$ and it is also quite obvious that
$$
 \fX_{\rm{CR}}\cap \fX_{\rm{CR}}^{*}=\fX_\R.
$$

\subsection{Manifold, ${\rm CR}$ and Complex Structures}\label{sec:X-structures}
We consider the following subsets of $\fX$. 
\begin{eqnarray}
 &&\label{eq:X'}
\fX'=\fX\setminus\fX_\R,\\
&&\label{eq:X''}
\fX''=\fX\setminus\fX_{\rm{CR}},\\
&&\label{eq:X*}
\fX^*=\fX'\setminus\fX_\C.
\end{eqnarray}
It has been proven in \cite{FP} that:
\begin{enumerate}
 \item $\fX'$ is a 4--dimensional real submanifold of $\C^3$,
 \item  $\fX''$ is a codimension 2 ${\rm CR}$ submanifold of $\C^3$ and
 \item $\fX^*$ is a 2--complex dimensional complex manifold.
\end{enumerate}
Below we are going to reprove these results, see Theorems \ref{thm:X'-sub},   \ref{thm:X''-CR} and \ref{thm:X*-J}, respectively. There is also a new result here: the ${\rm CR}$ structure is antiholomorphic, see Theorem \ref{thm:CRanti}. %Due to the different conventions for $\fX$ considered in \cite{FP} we are going to give sketches of the proofs in the next section.
 %We shall  discuss in some detail the nature of these singular sets in the Appendix. 

 \subsubsection{Manifold structure} \label{sec:submanifold}

\begin{thm}\label{thm:X'-sub}
The subset $\fX'=\fX\setminus\fX_\R$, where $\fX'$ and $\fX_\R$ are as in (\ref{eq:X'}) and (\ref{eq:XR}) respectively,  can be endowed with a structure of a 4--dimensional smooth real regular  submanifold of $\C^3$. %Here,
%$$
%\fX_{\R}=\left\{(\X_1,\X_2,\X_3)\in\fX\;|\; \X_i\in\R,\; \X_1+\X_2=1,\;\frac{1}{\X_2}+\frac{1}{\X_3}=1,\;\X_3+\frac{1}{\X_1}=1\right\}.
%$$
\end{thm}

\begin{proof}
 The proof is computational; one considers the equations defining $\fX$, those are rewritten as
\begin{eqnarray*}
\label{eq:F1}&&
F_1(\zeta_1,\zeta_2,\zeta_3)=|\zeta_2|^2-|\zeta_1|^2|\zeta_3|^2=0,\\
\label{eq:F2}&&
F_2(\zeta_1,\zeta_2,\zeta_3)=|\zeta_1|^2+|\zeta_2|^2-2\Re(\zeta_1+\zeta_2)+1-2|\zeta_1|^2\Re(\zeta_3)=0,
\end{eqnarray*}
where $\zeta_i=x_i+iy_i$, and calculates the rank of the Jacobian matrix
$$
D=\left[\begin{matrix} \frac{\partial F_1}{\partial x_1}&
\frac{\partial F_1}{\partial x_2}&
\frac{\partial F_1}{\partial x_3}&
\frac{\partial F_1}{\partial y_1}&
\frac{\partial F_1}{\partial y_2}&
\frac{\partial F_1}{\partial y_3}\\
\\
\frac{\partial F_2}{\partial x_1}&
\frac{\partial F_2}{\partial x_2}&
\frac{\partial F_2}{\partial x_3}&
\frac{\partial F_2}{\partial y_1}&
\frac{\partial F_2}{\partial y_2}&
\frac{\partial F_2}{\partial y_3}
\end{matrix}\right],
$$
at points of $\fX$. %The rest of the proof is computational; %We have
%\begin{eqnarray*}
%&&
%\frac{\partial F_1}{\partial x_1}=-2x_1|\zeta_3|^2,\quad
%\frac{\partial F_1}{\partial x_2}=2x_2,\quad
%\frac{\partial F_1}{\partial x_3}=-2x_3|\zeta_1|^2,\\
%&&
%\frac{\partial F_1}{\partial y_1}=-2y_1|\zeta_3|^2,\quad
%\frac{\partial F_1}{\partial y_2}=2y_2,\quad
%\frac{\partial F_1}{\partial y_3}=-2y_3|\zeta_1|^2,\\
%&&
%\frac{\partial F_2}{\partial x_1}=-4x_1\Re(\zeta_3)+2x_1-2,\quad
%\frac{\partial F_2}{\partial x_2}=2x_2-2,\quad
%\frac{\partial F_2}{\partial x_3}=2-|\zeta_1|^2,\\
%&&
%\frac{\partial F_2}{\partial y_1}=-4y_1\Re(\zeta_3)+2y_1,\quad
%\frac{\partial F_2}{\partial y_2}=2y_2,\quad
%\frac{\partial F_2}{\partial y_3}=0,
%\end{eqnarray*}
%and 
From the $2\times 2$ minor subdeterminants it eventually turns out that the rank is 2 everywhere except at points of $\fX_\R$. 
% From $d_{36}=0$ we have $y_3=0$ and from   $d_{25}=0$ we have $y_2=0$. Also $y_1=0$ from  $d_{14}=0$. Finally $x_1+x_2=1$  from $d_{13}=0$, $1/x_2+1/x_3=1$ from $d_{23}=0$ and $x_3+1/x_1=1$ from $d_{13}=0$.
The result now follows from the regular level set theorem.
\end{proof}

We  stress here that  $\fX'$ is maximal in the following sense: the diagonal action of ${\rm PU}(2,1)$ is free on $\mathfrak{F}_\R$ and not free on the subset $\mathfrak{F}'=\mathfrak{F}\setminus\mathfrak{F}_\R$ of the configuration space $\fF$. Therefore a natural (with respect to the group action) manifold structure can be given only in open  subsets of $\mathfrak{F}'$. Maximality now is in the sense that in fact  $\mathfrak{F}'$ is in bijection with $\fX'$ and thus inherits a manifold structure itself which is exactly the one defined in Theorem \ref{thm:X'-sub}. For details about the group action, see \cite{FP}.

\subsubsection{${\rm CR}$ structure}\label{sec:CR}

\begin{thm}\label{thm:X''-CR}
There is % subset $\fX''=\fX\setminus\fX_{\C\R}$ where $\fX''$ and $\fX_{\C\R}$ are as in (\ref{eq:X''}) and (\ref{eq:XCR}) respectively, can be endowed with 
a ${\rm CR}$ structure of codimension 2 defined on $\fX$. Its singular set is $\fX_{\rm{CR}}$.% which is given in (\ref{eq:XCR}).% where
%$$
%\fX_{\C\R}=\left\{(\X_1,\X_2,\X_3)\in\fX\;|\; \; \X_1+\X_2=1,\;\frac{1}{\X_2}+\frac{1}{\X_3}=1,\;\X_3+\frac{1}{\X_1}=1\right\}.
%$$
\end{thm}

\begin{proof}
Consider the defining equations \ref{eq:cross1} and \ref{eq:cross2} of $\fX$. Following the discussion in Section \ref{sec:cr}, %Section \ref{sec:CR-sub}, 
we examine whether the matrix
\begin{equation*}
D^{(1,0)}=\left[\begin{matrix}
\frac{\partial F_1}{\partial \zeta_1}&\frac{\partial F_1}{\partial \zeta_2}&\frac{\partial F_1}{\partial \zeta_3}\\
\\
\frac{\partial F_2}{\partial \zeta_1}&\frac{\partial F_2}{\partial \zeta_2}&\frac{\partial F_2}{\partial \zeta_3}
\end{matrix}\right]
\end{equation*}
has rank 2. %We have
%\begin{eqnarray*}
%&&
%\frac{\partial F_1}{\partial\zeta_1}=-\overline{\zeta_1}|\zeta_3|^2,\quad
%\frac{\partial F_1}{\partial\zeta_2}=\overline{\zeta_2},\quad
%\frac{\partial F_1}{\partial\zeta_3}=-\overline{\zeta_3}|\zeta_1|^2,\\
%&&
%\frac{\partial F_2}{\partial\zeta_1}=-2\overline{\zeta_1}\Re(\zeta_3)+\overline{\zeta_1}-1,\quad
%\frac{\partial F_2}{\partial\zeta_2}=\overline{\zeta_2}-1,\quad
%\frac{\partial F_2}{\partial\zeta_3}=-|\zeta_1|^2.
%\end{eqnarray*}
Calculating the $2\times 2$ minor subdeterminants $D_{\zeta_i,\zeta_j}=\left|\frac{\partial(F_1,F_2)}{\partial(\zeta_i,\zeta_j)}\right|$ at points of $\fX$ we obtain
\begin{eqnarray}
\label{eq:D23}&&
D_{\zeta_2,\zeta_3}=\frac{|\zeta_2|^2}{|\zeta_3|^2}\left(\overline{\zeta_2}\overline{\zeta_3}-\overline{\zeta_2}-\overline{\zeta_3}\right),\\
\label{eq:D31}&&
D_{\zeta_3,\zeta_1}=\overline{\zeta_3}|\zeta_1|^2\left(1+\overline{\zeta_1}\overline{\zeta_3}-\overline{\zeta_1}\right),\\
\label{eq:D12}&&
D_{\zeta_1,\zeta_2}=\frac{\overline{\zeta_2}}{\zeta_1}\left(1-\overline{\zeta_1}-\overline{\zeta_2}\right).
\end{eqnarray}
The $(1,0)-$vector field of $\C^3$
\begin{equation*}%\label{eq:ZCR}
Z=D_{\zeta_2,\zeta_3}\frac{\partial}{\partial \zeta_1}+D_{\zeta_3,\zeta_1}\frac{\partial}{\partial \zeta_2}+
D_{\zeta_1,\zeta_2}\frac{\partial}{\partial \zeta_3},
\end{equation*}
defined at points of $\fX''$, is the generator of $\calH^{(1,0)}$, that is, $\calH^{(1,0)}=\langle Z \rangle$.  The singular set $\calS$ of $\calH^{(1,0)}$ comprises points of $\fX'$ at which $Z$ is identically zero; 
it is clear that this happens only at points of $\fX_{\rm{CR}}$ and the proof is complete.
\end{proof}

\medskip

\begin{thm}\label{thm:CRanti}
 The ${\rm CR}$ structure of Theorem \ref{thm:X''-CR} is that of an antiholomorphic ${\rm CR}$ submanifold of $\C^3$.
\end{thm}
\begin{proof}
Consider the involution $\calT$ as in  (\ref{eq:T}) and let $W=\calT_*(Z)$. If $\omega=(\X_1,\X_2,\X_3)\in\fX$ we have
$$
W_\omega=\calT_*(Z)_\omega=\left(\calT_{*,\calT^{-1}(\omega)}\right)Z_{\calT^{-1}(\omega)}
$$
and thus
\begin{equation*}%\label{eq:WCR}
W=D_{\zeta_2,\overline{\zeta_3}}\frac{\partial}{\partial \zeta_1}+D_{\overline{\zeta_3},\zeta_1}\frac{\partial}{\partial \zeta_2}+
D_{\zeta_1,\zeta_2}\frac{\partial}{\partial \overline{\zeta_3}}.
\end{equation*}
The vector field $W$ is by definition in $T(\fX')$ and is nowhere zero at points of $\fX\setminus\fX_{\rm{CR}}^{*}$. We write
$$
Z=\frac{1}{2}\left(X-i\mathbb{J}X\right),\quad W=\frac{1}{2}\left(\calT_*(X)-i\calT_*(\mathbb{J}X)\right)=\frac{1}{2}\left(U-iV\right),
$$
where $\mathbb{J}$ is the natural complex structure of $\C^3$. Let also $\calH=\{X,Y=\mathbb{J}X\}$ and $\calV=\{U,V\}$; clearly $\calH\cap\calV=\{0\}$. From this, we also have that $\{X,Y,U,V\}$ is a basis for $\calH\oplus\calV$.% Now, $U$ and $V$ are in the tangent space for each such point and are not vanishing away from $\fX_{\C\R}^\perp$;  in fact we have for each $i=1,2$:
%\begin{eqnarray*}
% &&
% dF_i(U)=dF_i(\calT_*(X))=(\calT^{*}dF_i)(X)=0,
%\end{eqnarray*}
%where we have used the obvious relation $\calT^*(dF_i)=d(F_i\circ\calT)=dF_i$, $i=1,2$. In the same manner we prove that $V$ is also in the tangent space. To show that $\{X,Y,U,V\}$ is a basis  a basis of the tangent space, let $a_i$, $i=1,\dots,4$ such that
%$$
%a_1X+a_2Y+a_3U+a_4V=0.
%$$
%Applying $\calT_*$ we also have
%$$
%a_1U+a_2V+a_3X+a_4Y=0
%$$
%Assuming with no loss that $a_1\neq 0$, we solve the second equation w.r.t. $U$ and substitute to the second to obtain
%$$
%(a_3a_2-a_1a_4)V=(a_1^2-a_3^2)X+(a_1a_2-a_3a_4)Y.
%$$
%Now, all coefficients must vanish; this leads to
%$$
%X+Y+U+V=0,\quad\text{or}\quad X-Y-U-V=0,
%$$
%which is a contradiction since $\calH\cap\calV=\{0\}$. 

Finally, we show that $\mathbb{J}\calV\cap T(\fX')=\{0\}$. Indeed, for $i=1,2$ relations $dF_i(\mathbb{J}U)=0$ would imply $U-i\mathbb{J}U\in\calH^{(1,0)}$ and therefore $U\in\calH$, a contradiction. In the same manner we prove that $\mathbb{J}V$ is not in the tangent space and the proof is complete.

\end{proof}

The above antiholomorphic ${\rm CR}$ submanifold structure of cross--ratio variety makes sense away from points of $\fX_{\rm{CR}}\cup\fX_{\rm{CR}}^{*}$. In Section \ref{sec:ssagain} we show that this set is isomorphic with the subset of the configuration space $\fF$ consisting of quadruples $\fp=(p_1,p_2,p_3,p_4)$  such that either $p_1,p_2,p_3$ or $p_2,p_3,p_4$ lie in the same $\C-$circle.

\medskip

To complete this section, we calculate the Levi form $L=(L_1,L_2)$ of the above defined ${\rm CR}$ structure. After elementary calculations we find:
%\begin{eqnarray*}
%&&
%\frac{\partial^2 F_1}{\partial \zeta_1\partial{\overline \zeta_1}}=-|\zeta_3|^2,\quad
%\frac{\partial^2 F_1}{\partial \zeta_1\partial{\overline\zeta_2}}=0,\quad
%\frac{\partial^2 F_1}{\partial\zeta_1\partial{\overline \zeta_3}}=-\zeta_1\overline{\zeta_3},\\
%&&
%\frac{\partial^2 F_1}{\partial \zeta_2\partial{\overline \zeta_2}}=1,\quad \frac{\partial^2 F_1}{\partial \zeta_2\partial\overline{\zeta_3}}=0,\\
%&&
%\frac{\partial^2 F_1}{\partial \zeta_3\partial{\overline \zeta_3}}=-|\zeta_1|^2
%\end{eqnarray*}
%and also
%\begin{eqnarray*}
%&&
%\frac{\partial^2 F_2}{\partial \zeta_1\partial{\overline \zeta_1}}=1-2\Re(\zeta_3),\quad
%\frac{\partial^2 F_2}{\partial \zeta_1\partial\overline{\zeta_2}}=0,\quad
%\frac{\partial^2 F_2}{\partial\zeta_1\partial\overline{\zeta_3}}=-\zeta_1,\\
%&&
%\frac{\partial^2 F_2}{\partial \zeta_2\partial{\overline \zeta_2}}=1,\quad \frac{\partial^2 F_2}{\partial \zeta_2\partial{\overline \zeta_3}}=0,\\
%&&
%\frac{\partial^2 F_2}{\partial \zeta_3\partial{\overline \zeta_3}}=0.
%\end{eqnarray*}
%Therefore,
\begin{eqnarray*}
 &&
 L_1=|D_{\zeta_3,\zeta_1}|^2-|\zeta_1D_{\zeta_1,\zeta_2}+\zeta_3D_{\zeta_2,\zeta_3}|^2,\\
 &&
 L_2=(1-2\Re(\zeta_3))|D_{\zeta_2,\zeta_3}|^2+|D_{\zeta_3,\zeta_1}|^2-2\Re(\overline{\zeta_1}D_{\zeta_1,\zeta_2}D_{\overline{\zeta_2},\overline{\zeta_3}}).
\end{eqnarray*}
The following symmetric condition holds:
\begin{equation*}\label{eq:symmetric}
 \frac{1}{\zeta_1}D_{\zeta_2,\zeta_3}-\frac{1}{\zeta_2}D_{\zeta_3,\zeta_1}+\frac{1}{\zeta_3}D_{\zeta_1,\zeta_2}=0.
\end{equation*}
Using this, as well as the defining equations (\ref{eq:cross1}) and (\ref{eq:cross2}) of $\fX$, we deduce after elementary calculations that
$$
L_1\equiv 0,\quad L_2=|\zeta_1|^2|\zeta_2-\zeta_3\zeta_1|^2.
$$
At points of $\fX''$ we have $L_2>0$. Indeed, the only case where $L_2=0$ is when $\zeta_2=\zeta_3\zeta_1$. But this happens only at points of $\fX_{\rm{CR}}$, see Proposition 4.4 of \cite{FP}.

\subsection{Complex Structure $J$}\label{sec:complexJ}

The first one of the complex structures of $\fX^*$ we encounter in this work is revealed in the following theorem which has been proved in \cite{FP}. For clarity, we repeat here the proof.
\begin{thm}\label{thm:X*-J}
 The set $\fX^*$ can be endowed with the structure of a 2--complex dimensional complex manifold. With this structure $\fX^*$ is biholomorphic to $\C P^1\times(\C\setminus\R)$.
\end{thm}
\begin{proof}
 Pick a point $(\X_1,\X_2,\X_3)\in\fX^*$ and consider the unique class $[\fp]$ of quadruples $\fp=(p_1,p_2,p_3,$ $p_4)$ such that $\X_i(\fp)=\X_i$, $i=1,2,3$. We normalise so that
$$
p_1=(z_1,t_1),\quad p_2=\infty,\quad p_3=(0,0),\quad p_4=(z_4,t_4),\quad |z_1|+|z_4|\neq 0.
$$
We have
\begin{eqnarray*}
&&
\X_1=\frac{-|z_1|^2-it_1}{-|z_1|^2-|z_4|^2+i(t_4-t_1)+2\overline{z_1}z_4},\\
&&
\X_2=\frac{-|z_4|^2+it_4}{-|z_1|^2-|z_4|^2+i(t_4-t_1)+2\overline{z_1}z_4},\\
 &&
\X_3=\frac{-|z_4|^2+it_4}{-|z_1|^2+it_1},
\end{eqnarray*}
and the map
\begin{equation*}\label{eq:N}
\calN:\fX^*\ni(\X_1,\X_2,\X_3)\mapsto\left([z_1,z_4],\frac{|z_4|^2-it_4}{|z_1|^2-it_1}\right)\in\C P^1\times (\C-\R)
\end{equation*}
is a homeomorphism; with brackets we denote homogeneous coordinates in $\C P^1$. To define an atlas, we first observe that
\begin{equation*}
z=\frac{z_1}{z_4}=\frac{\X_1+\X_2/\X_3}{\X_1+\X_2-1},\quad  w=\frac{|z_4|^2-it_4}{|z_1|^2-it_1}=\X_3,
\end{equation*}
(the right equation is obvious; the left is obtained by straightforward calculations). Hence by considering
$\calN_0:\fX^*\to \C P^1\times (\C\setminus\R)$ and $\calN_\infty:\fX^*\to \C P^1\times (\C\setminus\R)$ given respectively by
$$
\calN_0(\X_1,\X_2,\X_3)=(z,w),\quad\text{and}\quad\calN_\infty(\X_1,\X_2,\X_3)=(1/z,w),
$$
we obtain an atlas $\calA_J$ for $\fX^*$, consisting of the charts $(\fX^*,\calN_0)$ and $(\fX^*,\calN_\infty)$.
Moreover, the complex manifold structure of $\fX^*$ which we shall denote by $J$, arises from this atlas.

\end{proof}

%DO YOU NEED THIS? FROM HERE

%We next prove a formula for the inverse mapping $\calN_0^{-1}$: 
%\begin{equation*}%\label{eq:N_0}
%\calN_0^{-1}(z,w)=\left(\frac{w|z|^2-1}{w-1+(w-|w|^2)|z|^2-(w-\bar w)\bar z},
%\frac{w-|w|^2|z|^2}{w-1+(w-|w|^2)|z|^2-(w-\bar w)\bar z},w\right).
%\end{equation*}
%To prove this, we first observe that by taking real and imaginary parts in both sides of the equation
%$$
%|z_4|^2-it_4=w(|z_1|^2-it_1)
%$$
%we get
%\begin{eqnarray*}
% &&
%t_1=\frac{|z_4|^2-\Re(w)|z_1|^2}{\Im(w)},\quad t_4=\frac{\Re(w)|z_4|^2-|w|^2|z_1|^2}{\Im(w)}.
%\end{eqnarray*}
%Thus
%\begin{eqnarray*}
% \X_1&=&\frac{|z_1|^2/|z_4|^2+it_1/|z_4|^2}{|z_1|^2/|z_4|^2+1+i(t_1-t_4)/|z_4|^2-2\overline{z_1}/\overline{z_4}}\\
%&=&\frac{|z|^2+i(1-\Re(w)|z|^2)/\Im(w)}{|z|^2+1+i(1-\Re(w))/\Im(w)+i(|w|^2-\Re(w))|z|^2/\Im(w)-2\overline{z}}\\
%&=&\frac{-iw|z|^2+i}{-i(w-\ovelin{w})(|z|^2+1)/2+i(1-(w+\overl}
%&=&\frac{w|z|^2-1}{w-1+(w-|w|^2)|z|^2-(w-\bar w)\bar z},
%\end{eqnarray*}
%and analogously for $\X_2$.

%\medskip

%TO HERE

\begin{cor}\label{cor:T-J}
 The involution $\calT$ of Equation (\ref{eq:T}) is an antiholomorphic mapping of the complex manifold $(\fX^*,J)$. 
\end{cor}

\begin{proof}
One verifies  that the coordinate expression of  $\calT$ is
 $$
(z,w)\mapsto\left(\frac{1}{\overline{wz}},\overline{w}\right),
$$
which is clearly antiholomorphic.
\end{proof}

\section{Pseudoconformality of Cross--Ratio Variety}\label{sec:Xpsc}
In this section we prove that away from certain singular sets, the cross--ratio variety can be given a psc as well as a spsc structure. In Section \ref{sec:complexI} we define the second complex operator for $\fX^*$ and finally in Section \ref{sec:main1}  we prove Theorems \ref{thm:X*-psc} and \ref{thm:X*-spsc}, respectively.

\subsection{Complex Structure $I$}\label{sec:complexI}
Let
$$
\fX^*_+=\{(\X_1,\X_2,\X_3)\in\fX^*\;|\;\Im(\X_3)>0\},\quad \fX^*_-=\{(\X_1,\X_2,\X_3)\in\fX^*\;|\;\Im(\X_3)<0\}.
$$
We consider the subset $\p$ of $\C^2$ defined as follows:
\begin{equation*}%\label{eq:P}
 \p=\left\{(\zeta_1,\zeta_2)\in\C^2\;|\;\left(|\zeta_1|-|\zeta_2|\right)^2<2\Re(\zeta_1)+2\Re(\zeta_2)-1\right\}.%-\p_s,
\end{equation*}
The set $\p$ is a {\it Levi strictly pseudoconvex domain}; the proof of this lies in Proposition \ref{prop:levipseudo}.

\medskip

We have the following Lemma:

\medskip

\begin{lem}\label{lem:M-surj}
Let $\calM:\fX\to \C^2$ be the  projection
$$
\fX\ni(\X_1,\X_2,\X_3)\mapsto(\X_1,\X_2)\in\C^2,
$$
and denote by $\calM_\pm$ the restrictions of $\calM$ to $\fX^*_\pm$ respectively. Then, $\calM_\pm$ are bijections of $\fX^*_\pm$ onto $\p$.
If $(\zeta_1,\zeta_2)\in\p$, then
\begin{equation}
 \calM_\pm^{-1}(\zeta_1,\zeta_2)=\left(\zeta_1,\zeta_2,\frac{|\zeta_2|}{|\zeta_1|}e^{\pm i\theta}\right)\in\fX^*_\pm,
\end{equation}
where
\begin{equation}\label{eq:theta}
\theta=\arccos\left(\frac{|\zeta_1|^2+|\zeta_2|^2-2\Re(\zeta_1)-2\Re(\zeta_2)+1}{2|\zeta_1||\zeta_2|}\right).
\end{equation}
\end{lem}

\medskip

\begin{proof}
Let $(\X_1,\X_2,\X_3)$ be any point in $\fX^*$. Since $\Im(\X_3)\neq 0$, from the obvious inequality
$$-|\X_3|<\Re(\X_3)<|\X_3|$$ and Equation (\ref{eq:cross2}) we get
$$
\left(|\X_1|-|\X_2|\right)^2<2\Re(\X_1)+2\Re(\X_2)-1<\left(|\X_1|+|\X_2|\right)^2,
$$
where the right inequality holds trivially. Thus $\calM_\pm(\fX^*_\pm)\subseteq\p$. Moreover
 $\theta$ is well defined, that is,
$$
-1< \frac{|\zeta_1|^2+|\zeta_2|^2-2\Re(\zeta_1)-2\Re(\zeta_2)+1}{2|\zeta_1||\zeta_2|}< 1.
$$
Writing
\begin{equation*}
 \X_1=\zeta_1,\quad \X_2=\zeta_2,\quad \X_3=\frac{|\zeta_2|}{|\zeta_1|}e^{\pm i\theta},
\end{equation*}
we have $\Im(\X_3)\neq 0$, 
$|\X_2|=|\X_1||\X_3|$ and
\begin{eqnarray*}
 2|\X_1|^2\Re(\X_3)&=&2|\X_1||\X_2|\cos\theta\\
&=&|\zeta_1|^2+|\zeta_2|^2-2\Re(\zeta_1)-2\Re(\zeta_2)+1\\
&=&|\X_1|^2+|\X_2|^2-2\Re(\X_1)-2\Re(\X_2)+1.
\end{eqnarray*}
The proof is complete.
\end{proof}

\medskip

\begin{thm}\label{thm:cs}
 The set $\fX^*$ can be endowed with the structure of a disconnected 2--complex dimensional complex manifold. The respective complex analytic atlas $\calA_I$ consists of the two non overlapping charts
$$
(\fX^*_+,\mathcal{M_+})\quad \text{and}\quad (\fX^*_-,\mathcal{M_-}),
$$
where ${\mathcal M_\pm}$ are the restrictions of $\calM$ to $\fX^*_\pm$, respectively. %and  $\iota$ is the complex conjugation in $\p$. 
In this manner, the complex structure in both  $\fX^*_+$ and $\fX^*_-$ is that of $\p$. %whereas the complex structure in  is that of $\overline{\p}$. 
\end{thm}

\medskip

The atlas $\calA_I$ of Theorem \ref{thm:cs} helps us to visualise the subset $\fX^*$ of $\fX$ as a disconnected set comprising of two 4--dimensional connected components, i.e., the sets $\fX^+$ and $\fX^-$. Both these sets are identified biholomorphically to $\p$. %and the second to $\overline{\p}$. 
From now on, the complex structure of $\fX^*$ induced from the complex analytic atlas above, will be denoted by $I$. 

\medskip

\begin{cor}\label{cor:T-I}
The involution $\calT$ given by Equation (\ref{eq:T}) is a holomorphic mapping of the complex manifold $(\fX^*,I)$. 
\end{cor}
\begin{proof}
 The coordinate expression for $\calT$ is
 $
 (\zeta_1,\zeta_2)\mapsto (\zeta_1,\zeta_2). 
 $
\end{proof}

\subsection{Proof of Theorems \ref{thm:X*-psc} and \ref{thm:X*-spsc}}\label{sec:main1}

In this section we are going to prove both Theorems \ref{thm:X*-psc} and \ref{thm:X*-spsc}. For the proof we need a series of lemmas.

\begin{lem}\label{lem:X*-psc}
The manifold $(\fX^*,I,J)$ is pseudoconformal with singular set $\fX_{\rm{CR}}\cap\fX^*$. 
\end{lem}
\begin{proof}
We prove that the identity map $id:(\fX^*,I)\to(\fX^*,J)$ is psc. %strictly pseudoconformal.
Working in the coordinate charts $(\fX^+,\calM^+)$ and $(\fX^*,\calN_0)$, we have the following representation of the identity map:
\begin{eqnarray*}
 (\zeta_1,\zeta_2)&\mapsto&(z,w) =\left(\frac{\zeta_1+\zeta_2/\zeta_3}{\zeta_1+\zeta_2-1}\;,\;\zeta_3\right),
\end{eqnarray*}
where $\zeta_3=\frac{|\zeta_2|}{|\zeta_1|}e^{i\theta}$ with $\theta$ as defined  in Equation (\ref{eq:theta}). From Corollary \ref{cor:psc} and Proposition \ref{prop:psc-man} we only have to prove
$$
Did.^{(0,1)}=\left[\begin{matrix}
                    \frac{\partial\overline{z}}{\partial \zeta_1}& \frac{\partial\overline{z}}{\partial \zeta_2}\\
\\
\frac{\partial\overline{w}}{\partial \zeta_1}& \frac{\partial\overline{w}}{\partial \zeta_2}
                   \end{matrix}\right]=0.
$$
Calculating straightforwardly we have:
\begin{eqnarray*}
 %&&
 %\frac{\partial z}{\partial\zeta_1}=-\frac{\zeta_2/\zeta_3^2}{\zeta_1+\zeta_2-1}\cdot\frac{\partial\zeta_3}{\partial \zeta_1}+\frac{\zeta_2-\zeta_2/\zeta_3-1}{(\zeta_1+\zeta_2-1)^2},\quad
 %\frac{\partial z}{\partial\zeta_2}=-\frac{\zeta_2/\zeta_3^2}{\zeta_1+\zeta_2-1}\cdot\frac{\partial\zeta_3}{\partial \zeta_2}+\frac{\zeta_1/\zeta_3-\zeta_1-1/\zeta_3}{(\zeta_1+\zeta_2-1)^2},\\
 %&&
%\frac{\partial w}{\partial\zeta_1}=\frac{\partial\zeta_3}{\partial\zeta_1},\quad
%\frac{\partial w}{\partial\zeta_2}=\frac{\partial\zeta_3}{\partial\zeta_2},\\
 &&
\frac{\partial\overline{z}}{\partial\zeta_1}=-\frac{\overline{\zeta_2}/\overline{\zeta_3}^2}{\overline{\zeta_1}+\overline{\zeta_2}-1}\cdot\frac{\partial\overline{\zeta_3}}{\partial\zeta_1},\quad
\frac{\partial\overline{z}}{\partial\zeta_2}=-\frac{\overline{\zeta_2}/\overline{\zeta_3}^2}{\overline{\zeta_1}+\overline{\zeta_2}-1}\cdot\frac{\partial\overline{\zeta_3}}{\partial\zeta_2},\\
&&
\frac{\partial\overline{w}}{\partial\zeta_1}=\frac{\partial\overline{\zeta_3}}{\partial\zeta_1},\quad
\frac{\partial\overline{w}}{\partial\zeta_2}=\frac{\partial\overline{\zeta_3}}{\partial\zeta_2}.
\end{eqnarray*}
%Thus it remains to calculate the partial derivatives $\partial\zeta_3/\partial\zeta_i$ and $\partial\overline{\zeta_3}/\partial\zeta_i$, $i=1,2$.  To do so, we take partial derivatives with respect to $\zeta_1$ and $\zeta_2$ in the equations
To calculate the partial derivatives $\partial\overline{\zeta_3}/\partial\zeta_i$, $i=1,2$ we take partial derivatives with respect to $\zeta_1$ and $\zeta_2$ in the equations
\begin{eqnarray*}
 &&
|\zeta_3|^2=\frac{|\zeta_2|^2}{|\zeta_1|^2},\\
&&
2|\zeta_1|^2\Re(\zeta_3)=|\zeta_1|^2+|\zeta_2|^2-2\Re(\zeta_1+\zeta_2)+1.
\end{eqnarray*}
We find the implicit expressions:
\begin{eqnarray}
 %&&\label{eq:zeta-derivatives}
%\frac{\partial\zeta_3}{\partial\zeta_1}=\frac{\zeta_3(\overline{\zeta_1}-\overline{\zeta_1}\zeta_3-1)}{2i\Im(\zeta_3)\cdot|\zeta_1|^2}=-\frac{D_{\overline{\zeta_3},\zeta_1}}{2i\Im(\zeta_3)\cdot|\zeta_1|^4},\quad\frac{\partial\zeta_3}{\partial\zeta_2}=-\frac{\zeta_3+\overline{\zeta_2}-\zeta_3\overline{\zeta_2}}{2i\Im(\zeta_3)\cdot|\zeta_1|^2}=\frac{D_{\zeta_2,\overline{\zeta_3}}}{2i\Im(\zeta_3)\cdot|\zeta_1|^4},\\
&&\label{eq:barzeta-derivatives}
\frac{\partial\overline{\zeta_3}}{\partial\zeta_1}=-\frac{\overline{\zeta_3}(\overline{\zeta_1}-\overline{\zeta_1}\overline{\zeta_3}-1)}{2i\Im(\zeta_3)\cdot|\zeta_1|^2}=\frac{D_{\zeta_3,\zeta_1}}{2i\Im(\zeta_3)\cdot|\zeta_1|^4},\quad \frac{\partial\overline{\zeta_3}}{\partial\zeta_2}=\frac{\overline{\zeta_2}+\overline{\zeta_3}-\overline{\zeta_2}\overline{\zeta_3}}{2i\Im(\zeta_3)\cdot|\zeta_1|^2}=-\frac{D_{\zeta_2,\zeta_3}}{2i\Im(\zeta_3)\cdot|\zeta_1|^4}.
\end{eqnarray} 
For furter use we shall also need the expressions:
\begin{eqnarray*}
 &&\label{eq:zeta-derivatives}
\frac{\partial\zeta_3}{\partial\zeta_1}=\frac{\zeta_3(\overline{\zeta_1}-\overline{\zeta_1}\zeta_3-1)}{2i\Im(\zeta_3)\cdot|\zeta_1|^2}=-\frac{D_{\overline{\zeta_3},\zeta_1}}{2i\Im(\zeta_3)\cdot|\zeta_1|^4},\quad\frac{\partial\zeta_3}{\partial\zeta_2}=-\frac{\zeta_3+\overline{\zeta_2}-\zeta_3\overline{\zeta_2}}{2i\Im(\zeta_3)\cdot|\zeta_1|^2}=\frac{D_{\zeta_2,\overline{\zeta_3}}}{2i\Im(\zeta_3)\cdot|\zeta_1|^4}.
\end{eqnarray*}
Here $D_{\zeta_2,\zeta_3}$, $D_{\zeta_3,\zeta_1}$ are as in Equations (\ref{eq:D23}) and (\ref{eq:D31}), respectively. %, and $D_{\zeta_2,\overline{\zeta_3}}$, $D_{\overline{\zeta_3},\zeta_1}$ are as the respective previous ones, but with $\zeta_3$ replaced by $\overline{\zeta_3}$. %Recall from Section \ref{sec:complexI} that
%$$
%\zeta_3=\frac{|\zeta_2|}{|\zeta_1|}e^{\pm i\theta},\quad \theta\;\text{as in Eq}\;\ref{eq:theta},
%$$
%where the sign depends on the respective chart.
%Let us now consider the matrices:
%$$
%Did.^{(0,1)}=\left[\begin{matrix}
 %                   \frac{\partial\overline{z}}{\partial \zeta_1}& \frac{\partial\overline{z}}{\partial \zeta_2}\\
%\\
%\frac{\partial\overline{w}}{\partial \zeta_1}& \frac{\partial\overline{w}}{\partial \zeta_2}
 %                  \end{matrix}\right]\quad\text{and}\quad
%Did.^{(1,0)}=\left[\begin{matrix}
 %                   \frac{\partial z}{\partial \zeta_1}& \frac{\partial z}{\partial \zeta_2}\\
%\\
%\frac{\partial w}{\partial \zeta_1}& \frac{\partial w}{\partial \zeta_2}
 %                  \end{matrix}\right].
%$$ 
Hence we have
\begin{eqnarray*}
\det\left(Did.^{(0,1)}\right)&=&\frac{\partial\overline{\zeta_3}}{\partial\zeta_1}\cdot \frac{\partial\overline{\zeta_3}}{\partial\zeta_2}\cdot\left|\begin{matrix}-\frac{\overline{\zeta_2}/\overline{\zeta_3}^2}{\overline{\zeta_1}+\overline{\zeta_2}-1} &-\frac{\overline{\zeta_2}/\overline{\zeta_3}^2}{\overline{\zeta_1}+\overline{\zeta_2}-1}\\
\\
1&1\end{matrix}\right|\\
&=0.
\end{eqnarray*}
We conclude the proof by remarking that from Equations (\ref{eq:barzeta-derivatives}) we also have that
the bundle $\calH^{(1,0)}(\fX^*,I)$ is generated by the vector field
 \begin{equation}\label{eq:ZI}
 Z^I=D_{\zeta_2,\zeta_3}\frac{\partial}{\partial\zeta_1}+D_{\zeta_3,\zeta_1}\frac{\partial}{\partial\zeta_2},
 \end{equation}
 away from points of $\fX^*\cap\fX_{\rm{CR}}$.
\end{proof}

\begin{lem}\label{lem:X*-spsc}
The psc structure of $(\fX^*,I,J)$ is spsc with singular set $(\fX_{\rm{CR}}\cup\fX_{\rm{CR}}^*)\cap\fX^*$.
\end{lem}
\begin{proof}
Since the involution $\calT$ is $I-$holomorphic and $J-$antiholomorphic, by Corollary \ref{cor:psc-spsc} the psc structure is spsc: The vertical bundle $\calV^{(1,0)}(\fX^*,I)$ is generated by the vector field
 \begin{equation}\label{eq:WI}
 W^I=\calT_*(Z^I)=D_{\zeta_2,\overline{\zeta_3}}\frac{\partial}{\partial\zeta_1}+D_{\overline{\zeta_3},\zeta_1}\frac{\partial}{\partial\zeta_2},
 \end{equation}
 away from points of $\fX^*\cap\fX_{\rm{CR}}^*$.
\end{proof}

\begin{lem}\label{X*-psc-spsc-sub}
 The following hold:
 \begin{enumerate}
\item $(\fX^*,I)$ is a pseudoconformal submanifold of $\C^3$ with singular set $\fX_{\rm{CR}}\cap\fX^*$. 
\item This psc structure is antiholomorphic with singular set $(\fX_{\C\R}\cup\fX_{\rm{CR}}^*)\cap\fX^*$. 
\item The induced ${\rm CR}$ and antiholomorphic ${\rm CR}$ submanifold structure coincide with the ${\rm CR}$ and antiholomorphic ${\rm CR}$ structure respectively, defined  in Section \ref{sec:CR} for the set $\fX''$.
\end{enumerate}
\end{lem}

\begin{proof}
 To prove (1), we consider the inclusion map  $\iota:(\fX^*,I)\hookrightarrow\C^3$;  this is  given in the chart $(\fX^*_+,\calM_+)$ by
\begin{equation*}
 \iota(\zeta_1,\zeta_2)=\left(\zeta_1,\zeta_2,\frac{|\zeta_2|}{|\zeta_1|}e^{i\theta}\right)=(\xi_1,\xi_2,\xi_3).
\end{equation*}
We will show first that $\iota_*\calH^{(1,0)}(M,I)$ is the restriction to $\fX^{*}$ of the ${\rm CR}$ structure defined in Section \ref{sec:CR}.
One checks that
$$
D\iota^{(0,1)}=\left[\begin{matrix}
0&0\\
\\
     0&0\\
\\
\frac{\partial\overline{\xi_3}}{\partial\zeta_1}&\frac{\partial\overline{\xi_3}}{\partial\zeta_2}                 
                     \end{matrix}\right]
$$
is clearly of rank 1, except at points where the partial derivatives of $\overline{\xi_3}$ vanish. From Equations (\ref{eq:barzeta-derivatives}) we have that this happens at points of $\fX^*\cap\fX_{\rm{CR}}$ and %To calculate these derivatives, we put $\xi_3=\zeta_3$ and after taking partial derivatives with respect to $\zeta_1$ and $\zeta_2$ in the equations
%\begin{eqnarray*}
% &&
%|\zeta_3|^2=\frac{|\zeta_2|^2}{|\zeta_1|^2},\\
%&&
%2|\zeta_1|^2\Re(\zeta_3)=|\zeta_1|^2+|\zeta_2|^2-2\Re(\zeta_1+\zeta_2)+1,
%\end{eqnarray*}
%we find
%\begin{eqnarray*}
% &&
%\frac{\partial\zeta_3}{\partial\zeta_1}=\frac{\zeta_3(\overline{\zeta_1}-\overline{\zeta_1}\zeta_3-1)}{2i\Im(\zeta_3)\cdot|\zeta_1|^2},\quad\frac{\partial\zeta_3}{\partial\zeta_2}=-\frac{\zeta_3+\overline{\zeta_2}-\zeta_3\overline{\zeta_2}}{2i\Im(\zeta_3)\cdot|\zeta_1|^2},\\
%&&
%\frac{\partial\overline{\zeta_3}}{\partial\zeta_1}=-\frac{\overline{\zeta_3}(\overline{\zeta_1}-\overline{\zeta_1}\overline{\zeta_3}-1)}{2i\Im(\zeta_3)\cdot|\zeta_1|^2},\quad \frac{\partial\overline{\zeta_3}}{\partial\zeta_2}=\frac{\overline{\zeta_2}+\overline{\zeta_3}-\overline{\zeta_2}\overline{\zeta_3}}{2i\Im(\zeta_3)\cdot|\zeta_1|^2}.
%\end{eqnarray*} 
thus $\iota$ is psc away from points of this set.

Recall from the proof of Lemma \ref{lem:X*-psc} that $\calH^{(1,0)}(\fX^*,I)$ is spanned in the chart $(\fX^+,\calM^+)$ by the vector field $Z^I$, where $Z^I$ is as in Equation (\ref{eq:ZI}).
%\begin{equation}
% Z^I=D_{\zeta_2,\zeta_3}\frac{\partial}{\partial \zeta_1}+D_{\zeta_3,\zeta_1}\frac{\partial}{\partial \zeta_2},
%\end{equation}
%where $D_{\zeta_2,\zeta_3}$ and $D_{\zeta_3,\zeta_1}$ are as in Eqs. \ref{eq:D23} and \ref{eq:D31} respectively, and $\zeta_3=|\zeta_2|e^{i\theta}/|\zeta_1|$. 
Using Equations (\ref{eq:D23}), (\ref{eq:D31}) and (\ref{eq:D12}) as well as Equations (\ref{eq:barzeta-derivatives}) and  (\ref{eq:zeta-derivatives}) we may verify the formulae:
\begin{eqnarray}
 &&\label{eqs:D1}
\frac{\partial\xi_3}{\partial\zeta_1}D_{\zeta_2,\zeta_3}+\frac{\partial\xi_3}{\partial\zeta_2}D_{\zeta_3,\zeta_1}=D_{\zeta_1,\zeta_2},
\\
&&\label{eqs:D2}
\frac{\partial\overline{\xi_3}}{\partial\zeta_1}D_{\zeta_2,\zeta_3}+\frac{\partial\overline{\xi_3}}{\partial\zeta_2}D_{\zeta_3,\zeta_1}=0.
\end{eqnarray}
%Here, $\partial\overline{\xi_3}/\partial\zeta_i$ and $\partial\xi_3/\partial\zeta_i$, $i=1,2$, are given by Eqs. (\ref{eq:barzeta-derivatives}) and  (\ref{eq:zeta-derivatives}), respectively (just replace $\zeta_3$ by $\xi_3$); Equation (\ref{eqs:D2}) follows immediately. %Now, the partial derivatives $\partial\xi_3/\partial\zeta_i$, $i=1,2$, are obtained in an analogous manner. One has:
%\begin{eqnarray}
 %&&\label{eq:zeta-derivatives}
%\frac{\partial\zeta_3}{\partial\zeta_1}=\frac{\zeta_3(\overline{\zeta_1}-\overline{\zeta_1}\zeta_3-1)}{2i\Im(\zeta_3)\cdot|\zeta_1|^2}=-\frac{D_{\overline{\zeta_3},\zeta_1}}{2i\Im(\zeta_3)\cdot|\zeta_1|^4},\quad\frac{\partial\zeta_3}{\partial\zeta_2}=-\frac{\zeta_3+\overline{\zeta_2}-\zeta_3\overline{\zeta_2}}{2i\Im(\zeta_3)\cdot|\zeta_1|^2}=\frac{D_{\zeta_2,\overline{\zeta_3}}}{2i\Im(\zeta_3)\cdot|\zeta_1|^4}.
%\end{eqnarray}
%Straightforward calculations and the use of Equations (\ref{eq:zeta-derivatives}) induce Equation (\ref{eqs:D1}).
%Now we claim that $\calH^{(1,0)}(\fX^*,I)$ is spanned in the chart $(\fX^+,\calM^+)$ by the vector field
%\begin{equation}
% Z^I=D_{\zeta_2,\zeta_3}\frac{\partial}{\partial \zeta_1}+D_{\zeta_3,\zeta_1}\frac{\partial}{\partial \zeta_2},
%\end{equation}
%where $D_{\zeta_2,\zeta_3}$ and $D_{\zeta_3,\zeta_1}$ are as in Eqs. \ref{eq:D23} and \ref{eq:D31} respectively, and $\zeta_3=|\zeta_2|e^{i\theta}/|\zeta_1|$. 
We next calculate:
\begin{eqnarray*}
 \iota_*Z^I&=&D_{\zeta_2,\zeta_3}\frac{\partial}{\partial \xi_1}+D_{\zeta_3,\zeta_1}\frac{\partial}{\partial \xi_2}+
\left(\frac{\partial\xi_3}{\partial\zeta_1}D_{\zeta_2,\zeta_3}+\frac{\partial\xi_3}{\partial\zeta_2}D_{\zeta_3,\zeta_1}\right)\frac{\partial}{\partial \xi_3}\\
&&+\left(\frac{\partial\overline{\xi_3}}{\partial\zeta_1}D_{\zeta_2,\zeta_3}+\frac{\partial\overline{\xi_3}}{\partial\zeta_2}D_{\zeta_3,\zeta_1}\right)\frac{\partial}{\partial \overline{\xi_3}}\\
\text{using  (\ref{eqs:D1}) and (\ref{eqs:D2})}&=&D_{\zeta_2,\zeta_3}\frac{\partial}{\partial \xi_1}+D_{\zeta_3,\zeta_1}\frac{\partial}{\partial \xi_2}+D_{\zeta_1,\zeta_2}\frac{\partial}{\partial \xi_3},
\end{eqnarray*}
which proves our assertion. To prove (2) and (3) we consider the vector field $W^I$ as in (\ref{eq:WI}). The following relations hold: 
\begin{eqnarray}
 &&\label{eqs:D3}
\frac{\partial\xi_3}{\partial\zeta_1}D_{\zeta_2,\overline{\zeta_3}}+\frac{\partial\xi_3}{\partial\zeta_2}D_{\overline{\zeta_3},\zeta_1}=0,
\\
&&\label{eqs:D4}
\frac{\partial\overline{\xi_3}}{\partial\zeta_1}D_{\zeta_2,\overline{\zeta_3}}+\frac{\partial\overline{\xi_3}}{\partial\zeta_2}D_{\overline{\zeta_3},\zeta_1}=D_{\zeta_1,\zeta_2}.
\end{eqnarray}
We therefore have:
\begin{eqnarray*}
 \iota_*W^I&=&D_{\zeta_2,\overline{\zeta_3}}\frac{\partial}{\partial \xi_1}+D_{\overline{\zeta_3},\zeta_1}\frac{\partial}{\partial \xi_2}+
\left(\frac{\partial\xi_3}{\partial\zeta_1}D_{\zeta_2,\overline{\zeta_3}}+\frac{\partial\xi_3}{\partial\zeta_2}D_{\overline\zeta_3,\zeta_1}\right)\frac{\partial}{\partial \xi_3}\\
&&+\left(\frac{\partial\overline{\xi_3}}{\partial\zeta_1}D_{\zeta_2,\overline{\zeta_3}}+\frac{\partial\overline{\xi_3}}{\partial\zeta_2}D_{\overline{\zeta_3},\zeta_1}\right)\frac{\partial}{\partial \overline{\xi_3}}\\
\text{using  (\ref{eqs:D3}) and (\ref{eqs:D4})}&=&D_{\zeta_2,\overline{\zeta_3}}\frac{\partial}{\partial \xi_1}+D_{\overline{\zeta_3},\zeta_1}\frac{\partial}{\partial \xi_2}+D_{\zeta_1,\zeta_2}\frac{\partial}{\partial \overline{\xi_3}},
\end{eqnarray*}
and the proof is complete.
\end{proof}

\section{Geometric Interpretation of the Involution $\calT$}\label{sec:tg}

In this section we are going to prove a theorem which sheds light into the mysterious nature of the involution $\calT$. Speaking in terms of the isomorphism $\varpi$ of the configuration space $\fF$ and the cross--ratio variety $\fX$, if $[\fp]\in\fF$ such that not all points of $\fp$ lie in the same $\C-$circle, then $(\varpi^{-1}\circ\calT\circ\varpi)([\fp])$ is the class $[\fp']$, where $\fp'$ is the quadruple of pairwise distinct points which are obtained after applying at points of $\fp$ certain elements of ${\rm PU}(2,1)$ which are congruent to Heisenberg similarities depending only on $\X_i$, $i=1,2,3$, where $(\X_1,\X_2,\X_3)=\varpi([\fp])$. More precisely, we have:

\begin{thm}\label{thm:giT}
Let $\fp=(p_1,p_2,p_3,p_4)$ and $\fp'=(p'_1,p'_2,p'_3,p'_4)$ be two quadruples of pairwise distinct points in $\partial\bH^2_\C$ with respective cross ratios $\X_i$ and $\X'_i$, $i=1,2,3$. Assume that: 
\begin{enumerate}
 \item [{a)}] Neither the quadruple $\fp$ nor the quadruple $\fp'$ belongs to a $\C-$circle.
 \item [{b)}] $\Im(\X_3)$ and $\Im(\X'_3)$ are both different from zero. 
\end{enumerate}

Then $\mathcal{T}(\X_1,\X_2,\X_3)=(\X_1',\X_2',\X_3')$ if and only if 
there exist holomorphic isometries $g_1$, $g_4$ %and antiholomorphic isometries $h_1$, $h_4$ 
of $\partial\bH^2_\C$ such that:
\begin{enumerate}
 \item   $g_i(p_2)=p'_2$, $g_i(p_3)=p'_3$ for $i=1,4$, $g_1(p_1)=p'_4$, $g_4(p_4)=p'_1$; 
\item the composition $g_1\circ g_4$ is conjugate to the rotation in  an angle $\arg(\X_3)$;
\item the composition $g_1\circ g_4^{-1}$ is conjugate to the dilation by $|\X_3|$ followed by the rotation in an angle $2\arg\left(\frac{1-\X_1-\X_2}{\X_1^{1/2}\X_2^{1/2}}\right)$.
%\item $h_i(p_2)=p_3'$, $h_i(p_3)=p'_2$ for $i=1,4$, $h_1(p_1)=p'_4$, $h_4(p_4)=p'_1$ and the composition $h_4\circ h_1$ is conjugate to the rotation in  angle $\arg(\X_3)$.
\end{enumerate}
%\begin{enumerate}
 %\item $g_1(p_1)=p'_4$, $g_1(p_2)=p'_2$, $g_1(p_3)=p'_3$,
%\item   $g_4(p_2)=p'_2$, $g_4(p_3)=p'_3$, $g_4(p_4)=p'_1$ and 
%\item  the composition $g_1\circ g_4$ is conjugate to the rotation in an angle $\arg(\X_3)$.
%\end{enumerate} 
\end{thm}

For the proof, we need some preliminary discussion first. Associated to a quadruple of points $\fp=(p_1,p_2,p_3,p_4)$ are the following Cartan's angular invariants:
\begin{equation*}\label{eq:Ai}
\A_1(\fp)=\A(p_2,p_3,p_4),\quad \A_2(\fp)=\A(p_1,p_3,p_4),\quad \A_3(\fp)=\A(p_1,p_2,p_4),\quad \A_4(\fp)=\A(p_1,p_2,p_3).
\end{equation*}

The next proposition shows the connection of $\A_i(\fp)$ with the cross--ratios $\X_i(\fp)$ as well as the connection between them.

\begin{prop}\label{prop:X-A}
Let $\fp=(p_1,p_2,p_3,p_4)$ a quadruple of pairwise distinct points in $\partial\bH^2_\C$, $\X_i=\X_i(\fp)$ $i=1,2,3$ and $\A_i=\A_i(\fp)$. Then
\begin{eqnarray*}
&&
\label{eq:X-A}
\arg(\X_1)=\A_1-\A_2,\quad \arg(\X_2)=-\A_2-\A_4,\quad \arg(\X_3)=\A_4-\A_1,\\
&&\label{eq:As}
\A_3=\A_2-\A_1+\A_4.
\end{eqnarray*}
\end{prop}

\begin{proof}
We only prove the first of these identities; the proof of the rest is similar and it is left  to the reader. We have:
\begin{eqnarray*}
\arg\X_1&=&\arg\frac{\langle\bp_4,\bp_2\rangle \langle\bp_3,\bp_1\rangle}{\langle\bp_4,\bp_1\rangle \langle\bp_3,\bp_2\rangle}\\
&=&\arg\frac{\langle\bp_4,\bp_2\rangle \langle\bp_2,\bp_3\rangle\langle\bp_3,\bp_1\rangle \langle\bp_1,\bp_3\rangle}{\langle\bp_4,\bp_1\rangle \langle\bp_2,\bp_3\rangle\langle\bp_3,\bp_2\rangle \langle\bp_1,\bp_3\rangle}\\
&=&\arg\frac{\langle\bp_4,\bp_2\rangle \langle\bp_2,\bp_3\rangle\left| \langle\bp_1,\bp_3\rangle\right|^2}{\langle\bp_1,\bp_3\rangle\langle\bp_4,\bp_1\rangle \left|\langle\bp_2,\bp_3\rangle\right|^2 }\\
&=&\arg\frac{\langle\bp_2,\bp_3\rangle\langle\bp_3,\bp_4\rangle\langle\bp_4,\bp_2\rangle }{\langle\bp_1,\bp_3\rangle\langle\bp_3,\bp_4\rangle\langle\bp_4,\bp_1\rangle  }\\
&=&\arg\left(-\langle\bp_2,\bp_3\rangle\langle\bp_3,\bp_4\rangle\langle\bp_4,\bp_2\rangle\right)-
\arg\left(-\langle\bp_1,\bp_3\rangle\langle\bp_3,\bp_4\rangle\langle\bp_4,\bp_1\rangle\right)\\
&=&\A_1-\A_2.
\end{eqnarray*}
\end{proof}

We now need the following two lemmas.

\begin{lem}\label{lem:giT}
With the assumptions of Theorem \ref{thm:giT}, let also 
%Let $\fp=(p_1,p_2,p_3,p_4)$ and $\fp'=(p'_1,p'_2,p'_3,p'_4)$ be two quadruples of distict points in $\partial\bH^2_\C$ with respective cross ratios $\X_i$ and $\X'_i$, $i=1,2,3$ and respective Cartan's invariants 
$\A_i$ and $\A'_i$, $i=1,2,3,4,$ be the respective Cartan's angular invariants of $\fp$ and $\fp'$. Let also
$$
2\eta=\arg(1-\X_1-\X_2),\quad 2\eta'=\arg(1-\X'_1-\X'_2). 
$$
Then, the following are equivalent:
\begin{enumerate}
 \item[{i)}] $\mathcal{T}(\X_1,\X_2,\X_3)=(\X_1',\X_2',\X_3').$
\item [{ii)}]%\begin{eqnarray}
%&&\label{eq:conj-cond1}
$|\X_i|=|\X'_i|$, $i=1,2,$ and $\A_1=\A'_4,\; \A_2=\A'_3,\; \A_3=\A'_2,\; \A_4=\A'_1$.
%\end{eqnarray}
\item [{iii)}] %\begin{eqnarray}
%&&\label{eq:conj-cond2}
$|\X_3|=|\X'_3|$  and $\eta=\eta'$,\; $\A_1=\A'_4$,\; $\A_2=\A'_3$,\; $\A_3=\A'_2$,\; $\A_4=\A'_1$.
%\end{eqnarray}
\end{enumerate}

\end{lem}

%\medskip

\begin{proof}
%Recall  from Proposition \ref{prop:X-A} that the following hold:
%\begin{eqnarray*}
%&&
%\arg(\X_1)=\A_1-\A_2,\quad \arg(\X_2)=-\A_2-\A_4,\quad \arg(\X_3)=\A_3-\A_2,\\
%&&
%\A_1=\A_2-\A_3+\A_4.
%\end{eqnarray*}
We first prove direction i) $\Rightarrow$ iii). Since $\X_i=\X_i'$, $i=1,2$ and $\X_3'=\overline{\X_3}$ we  clearly have $|\X_3|=|\X'_3|$ and also
$
2\eta=%\arg(1-\X_1-\X_2)=\arg(1-\X'_1-\X'_2)=
2\eta'. 
$
Now, from the relations $\arg(\X_i)=\arg(\X'_i)$, $i=1,2$ and $\arg(\X_3)=-\arg(\X'_3)$, we also have from Proposition \ref{prop:X-A} that
\begin{eqnarray*}
&&
\A_1-\A_2=\A'_1-\A'_2,\\
&&
\A_2+\A_4=\A'_2+\A'_4,\\
&&
\A_3-\A_2=-\A'_3+\A_2.
\end{eqnarray*}
Again by Proposition \ref{prop:X-A}, the third equation can be replaced by the equivalent $\A_4-\A_1=\A_1'-\A'_4$; solving the $3\times 3$ system in $\A_1,\A_2$ and $\A_4$ we get
$$
\A_1=\A_4',\quad \A_4=\A_1',\quad \A_2=\A_2'+\A'_4-\A'_1=\A'_3.
$$
Hence this also yields
$$
\A_3=\A_2+\A_4-\A_1=\A'_3+\A'_1-\A'_4=\A'_2.
$$
To  prove direction iii) $\Rightarrow$ ii)  we first show that
\begin{equation}\label{eq:eta}
e^{2i\eta}=\frac{1-\X_1-\X_2}{|1-\X_1-\X_2|},\quad\text{and}\quad 2|\X_1|^{1/2}|\X_2|^{1/2}\sqrt{\cos(\A_1)\cos(\A_4)}=|1-\X_1-\X_2|.
\end{equation}
The left equation is following from the definition of $\eta$. As for the right equation, observe that
\begin{eqnarray*}
|1-\X_1-\X_2|^2&=&1+|\X_1|^2+|\X_2|^2-2\Re(\X_1)-2\Re(\X_2)+2\Re(\X_1\overline{\X_2})\\
&=&2|\X_1|^2\Re(\X_3)+2\Re(\X_1\overline{\X_2})\\
&=&2|\X_1||\X_2|\left(\cos(\arg(\X_3))+\cos(\arg(\X_1\overline{\X_2}))\right)\\
&=&2|\X_1||\X_2|\left(\cos(\A_4-\A_1)+\cos(\A_1+\A_4))\right)\\
&=&4|\X_1||\X_2|\cos(\A_1)\cos(\A_4).
\end{eqnarray*}
Therefore, $\eta=\eta'$  implies
$$
e^{2i\eta}=\frac{1-\X_1-\X_2}{2|\X_1|^{1/2}|\X_2|^{1/2}\sqrt{\cos(\A_1)\cos(\A_4)}}=
\frac{1-\X'_1-\X'_2}{2|\X'_1|^{1/2}|\X'_2|^{1/2}\sqrt{\cos(\A'_1)\cos(\A'_4)}}=e^{2i\eta'}.
$$
Since we additionally have $|\X_3|=|\X'_3|$, $\A_1=\A'_4$ and $\A_4=\A'_1$, this is equivalent to
$$
\frac{1-\X_1-\X_2}{|\X_1|}=
\frac{1-\X'_1-\X'_2}{|\X'_1|}.
$$
From the conditions on Cartan's angular invariants we have $\arg{\X_1}=\arg{\X'_i}$, $i=1,2$;  the above equation then implies
$$
\frac{1}{|\X_1|}-e^{i\arg(\X_1)}-|\X_3|e^{i\arg(|\X_2|)}=\frac{1}{|\X'_1|}-e^{i\arg(\X_1)}-|\X_3|e^{i\arg(|\X_2|)}.
$$
Hence $|\X_1|=|\X'_1|$ and therefore also $|\X_2|=|\X'_2|$.

Finally, we prove direction ii) $\Rightarrow$ i).
\begin{eqnarray*}
 &&
 \X_1'=|\X_1'|e^{i\arg(\X_1')}=|\X_1|e^{i(\A_1'-\A_2')}=|\X_1|e^{i(\A_4-\A_3)}=|\X_1|e^{i(\A_1-\A_2)}=|\X_1|e^{i\arg(\X_1)}=\X_1,\\
 &&
 \X_2'=|\X_2'|e^{i\arg(\X_2')}=|\X_1|e^{-i(\A_2'+\A_4')}=|\X_2|e^{-i(\A_2+\A_1)}=|\X_2|e^{-i(\A_2+\A_4)}=|\X_2|e^{i\arg(\X_2)}=\X_2,\\
 &&
 \X_3'=|\X_3'|e^{i\arg(\X_3')}=|\X_3|e^{i(\A_4'-\A_1')}=|\X_3|e^{i(\A_1-\A_4)}=|\X_3|e^{-i\arg(\X_3)}=\overline{\X_3},
\end{eqnarray*}
The proof is complete.
\end{proof}

\begin{lem}\label{lem:norm}
 Let $\fp=(p_1,p_2,p_3,p_4)$ a quadruple of pairwise distinct points of $\partial\bH^2_\C$ with cross--ratios $\X_1,\X_2,\X_3$ and assume that  $\fp$ does not belong to a $\C-$circle and also that $\Im(\X_3)$ is different from zero. %such that $(\X_1,\X_2)\in\p^*$.
Then we can normalise so that 
\begin{eqnarray*}
&& 
p_1=(z_1,t_1)=\left(\sqrt{\cos(\A_4)}e^{-i\A_3},\;\sin(\A_4)\right),\\
&&
p_2=\infty,\quad p_3=(0,0),\\
&&
p_4=(z_4,t_4)=\left(-\sqrt{\cos(\A_1)}|\X_3|^{1/2}e^{2i\eta},\;|\X_3|\sin(\A_1)\right).
\end{eqnarray*}
Here $\A_i$, $i=1,\dots,4$ are the Cartan's angular invariants of $\fp$ and $2\eta=\arg(1-\X_1-\X_2)$.
\end{lem}

\medskip

\begin{proof}
By applying a suitable element of ${\rm PU}(2,1)$ we may normalise so that $p_i$  have  lifts
\begin{eqnarray*}
 &&
\bp_1=\left[\begin{matrix}
-e^{-i\A_4}\\ \sqrt{2\cos(\A_4)}e^{-i\A_3}\\ 1
            \end{matrix}\right],\quad \bp_2=\left[\begin{matrix}
1\\ 0\\ 0
            \end{matrix}\right],\quad
\bp_3=\left[\begin{matrix}
0\\ 0\\ 1
            \end{matrix}\right],\quad \bp_4=\left[\begin{matrix}
-|\X_3|e^{-i\A_1}\\ -\sqrt{2\cos(\A_1)}|\X_3|^{1/2}e^{2i\eta}\\ 1
            \end{matrix}\right],
\end{eqnarray*}
(compare with the expression in the proof of Proposition 5.5 in \cite{PP}).
We have
\begin{eqnarray*}
&&
\langle\bp_1,\bp_2\rangle=\langle\bp_2,\bp_3\rangle=\langle\bp_2,\bp_4\rangle=1,\\
&&
\langle\bp_1,\bp_3\rangle=-e^{-i\A_4},\quad \langle\bp_3,\bp_4\rangle=-|\X_3|e^{i\A_1}.
\end{eqnarray*}
Also,
\begin{eqnarray*}
\langle\bp_1,\bp_4\rangle&=&-e^{-i\A_4}-|\X_3|e^{i\A_1}-2\sqrt{\cos(\A_1)\cos(\A_4)}|\X_3|^{1/2}e^{-2i\eta}\cdot e^{-i\A_3}\\
&=&-e^{-i\A_4}-\frac{|\X_2|}{|\X_1|}e^{i\A_1}-2\sqrt{\cos(\A_1)\cos(\A_4)}\frac{|\X_2|^{1/2}}{|\X_1|^{1/2}}e^{-2i\eta}\cdot e^{-i\A_3}\\
&=&-e^{-i\A_4}-\frac{|\X_2|}{|\X_1|}e^{i\A_1}-\frac{|1-\X_1-\X_2|}{|\X_1|^{1/2}|\X_2|^{1/2}}\cdot\frac{|\X_2|^{1/2}}{|\X_1|^{1/2}}\cdot\frac{1-\overline{\X_1}-\overline{\X_2}}{|1-\X_1-\X_2|}\cdot e^{-i\A_3}\\
&=&-\frac{|\X_1|e^{i(\A_3-\A_4)}+|\X_2|e^{i(\A_1+\A_3)}+1-\overline{\X_1}-\overline{\X_2}}{|\X_1|e^{i\A_3}}\\
&=&-\frac{\overline{\X_1}+\overline{\X_2}+1-\overline{\X_1}-\overline{\X_2}}{|\X_1|e^{i\A_3}}\\
&=&-|\X_1|^{-1}e^{-i\A_3},
\end{eqnarray*}
where we have used equations (\ref{eq:eta}).
%$$
%e^{-2i\eta}=\frac{1-\overline{\X_1}-\overline{\X_2}}{|1-\X_1-\X_2|},\quad\text{and}\quad 2|\X_1|^{1/2}|\X_2|^{1/2}\sqrt{\cos(\A_1)\cos(\A_4)}=|1-\X_1-\X_2|.
%$$
Therefore we get
\begin{eqnarray*}
&&
\frac{\langle\bp_4,\bp_2\rangle\langle\bp_3,\bp_1\rangle}{\langle\bp_4,\bp_1\rangle\langle\bp_3,\bp_2\rangle}=
\frac{-e^{i\A_4}}{-|\X_1|^{-1}e^{i\A_3}}=\X_1,\\
&&
\frac{\langle\bp_4,\bp_3\rangle\langle\bp_2,\bp_1\rangle}{\langle\bp_4,\bp_1\rangle\langle\bp_2,\bp_3\rangle}=
\frac{-|\X_3|e^{-i\A_1}}{-|\X_1|^{-1}e^{i\A_3}}=\X_2,\\
&&
\frac{\langle\bp_4,\bp_3\rangle\langle\bp_2,\bp_1\rangle}{\langle\bp_4,\bp_2\rangle\langle\bp_1,\bp_3\rangle}=
\frac{-|\X_3|e^{-i\A_1}}{-e^{-i\A_4}}=\X_3.
\end{eqnarray*}
\end{proof}

\medskip

\noindent{\it Proof of Theorem \ref{thm:giT}.}

\noindent Using Lemma \ref{lem:norm} we normalise so that
\begin{eqnarray*}
&& 
p_1=(z_1,t_1)=\left(\sqrt{\cos(\A_4)}e^{-i\A_3},\;\sin(\A_4)\right),\\
&&
p_2=\infty,\quad p_3=(0,0),\\
&&
p_4=(z_4,t_4)=\left(-\sqrt{\cos(\A_1)}|\X_3|^{1/2}e^{2i\eta},\;|\X_3|\sin(\A_1)\right).
\end{eqnarray*}
%and also,
%\begin{eqnarray*}
%&& 
%p'_1=(z'_1,t'_1)=\left(\sqrt{\cos(\A'_4)}e^{-i\A'_3},\;\sin(\A'_4)\right),\\
%&&
%p'_2=\infty,\quad p'_3=(0,0),\\
%&&
%p'_4=(z'_4,t'_4)=\left(-\sqrt{\cos(\A'_1)}|\X'_3|^{1/2}e^{2i\eta'},\;|\X'_3|\sin(\A'_1)\right).
%\end{eqnarray*}
Suppose first that $\mathcal{T}(\X_1,\X_2,\X_3)=(\X'_1,\X'_2,\X'_3)$. Then, using iii) of Lemma \ref{lem:giT} we have the following normalisation for $p_i'$, $i=1,\dots,4$:
\begin{eqnarray*}
&& 
p'_1=(z'_1,t'_1)=\left(\sqrt{\cos(\A_1)}e^{-i\A_2},\;\sin(\A_1)\right),\\
&&
p'_2=\infty,\quad p'_3=(0,0),\\
&&
p'_4=(z'_4,t'_4)=\left(-\sqrt{\cos(\A_4)}|\X_3|^{1/2}e^{2i\eta},\;|\X_3|\sin(\A_4)\right).
\end{eqnarray*}
It follows after elementary calculations that the transformations
\begin{eqnarray*}
 &&
g_1(z,t)=\left(-|\X_3|^{1/2}e^{i(2\eta+\A_3)}z,|\X_3|t\right),\quad g_4(z,t)=\left(-|\X_3|^{-1/2}e^{-i(2\eta+\A_2)}z,|\X_3|^{-1}t\right),
\end{eqnarray*}
 satisfy conditions  (1), (2) and (3) of the theorem.

Conversely, suppose that there exist $g_i$, $i=1,4$ as in the conditions of the theorem. We  write
\begin{eqnarray*}
&&
 g_1(z,t)=\left(e^{l_1+3i\theta_1} z,e^{2l_1}t\right),\\
&&
g_4(z,t)=\left(e^{l_4+3i\theta_4} z,e^{2l_4}t\right). 
\end{eqnarray*}
From (1), and since $g_1$ maps $(p_1,p_2,p_3)$ to $(p'_4,p'_2,p'_3)$ we have $\A_4=\A'_1$. Also, since $g_4$ maps $(p_2,p_3,p_4)$ to $(p'_2,p'_3,p'_1)$ we have $\A_1=\A'_4$. 
By our assumptions, 
\begin{eqnarray*}
&&
l_1=\log(|\X_3|)^{1/2},\quad l_4=\log(|\X_3|)^{-1/2},\\
&&
3\theta_1=2\eta+\A_3,\quad 3\theta_4=-2\eta-\A_2. 
\end{eqnarray*}
Under our normalisation assumptions, $g_1(z_1,t_1)=(z'_4,t'_4)$ gives $2\eta=2\eta'\;\mod(\pi)$. Using this, $g_4(z_4,t_4)=(z'_1,t'_1)$ gives $\A_2=\A_3'$ and therefore also $\A'_3=\A_2$. The result now follows from iii) of Lemma \ref{lem:giT}.

\hfill{$\Box$}

\section{Further Comments on the Cross--Ratio Variety}\label{sec:app}
This section contains further supplementary information about the cross--ratio variety $\fX$. Using the Cartan's angular invariants associated to certain triples of a given quadruple $\fp$ of pairwise distinct points, we are able to give an alternative description of $\fX$ in Section \ref{sec:alternative}. Section \ref{sec:ssagain} is concerned with the analytic description of singular sets; the complex singular set $\fX_\C$ which is the largest of all singular sets and the one with the richer structure, is studied in Section  \ref{sec:csset}. %Finally, in Section \ref{sec:tg} we  give a geometric interpretation of the involution $\mathcal{T}$.

\subsection{An alternative description of $\fX$}\label{sec:alternative}
We may give an alternative description of $\fX$ using the results of Proposition \ref{eq:X-A}. Plugging those into Equation (\ref{eq:cross2}), we immediately have a counterpart of Proposition \ref{prop:cross-ratio-equalities}, as an alternative definition of the cross--ratio variety $\fX$.

%\medskip

\begin{prop}\label{prop:cross-ratio-inequalities-2}
Let $\fp=(p_1,p_2,p_3,p_4)$ be any quadruple of pairwise distinct points in $\partial \bH^2_\C$. Let
$$
\X_1(\fp)=\X(p_1,p_2,p_3,p_4),\quad \X_2(\fp)=\X(p_1,p_3,p_2,p_4),
$$
and
$$
\A_1(\fp)=\A(p_2,p_3,p_4),\quad \A_2(\fp)=\A(p_1,p_3,p_4),\quad \A_4(\fp)=\A(p_1,p_2,p_3).
$$
Then
\begin{eqnarray*}%%\label{eq:cross1'}
%%|\X_2|&=&|\X_1||\X_3|,\\
\label{eq:cross2'}
|\X_1|^2+|\X_2|^2&=&2|\X_1||\X_2|\cos(\A_1-\A_4)\\
\notag &&+2|\X_1|\cos(\A_2-\A_1)+2|\X_2|\cos(\A_2+\A_4)-1.
%\label{eq:cross3'}
%\A_3&=&\A_2-\A_1+\A_4.
\end{eqnarray*}
\end{prop}

In this manner we obtain a description of $\fX$ as a 4--dimensional real subset of $\R^2_+\times [-\pi/2,$ $\pi/2]^3$, compare to the one in \cite{CG}.
We underline that the choice of $\A_1,\A_2$ and $\A_4$ is arbitrary; Proposition \ref{prop:cross-ratio-inequalities-2} can be modified analogously for any other choice of three Cartan's angular invariants among $\A_i(\fp)$, $i=1,2,3,4$.

\medskip

The next lemma which relates cross--ratios and Cartan's angular invariants will be useful in our subsequent discussion.

\begin{lem}\label{lem:X-A}
The following formulae fold.
\begin{eqnarray*}
&&
|\X_1+\X_2-1|^2=4|\X_1||\X_2|\cos(\A_1)\cos(\A_4),\\
&&
|\X_1+\overline{\X_2}-1|^2=4|\X_1||\X_2|\cos(\A_2)\cos(\A_3),
\end{eqnarray*}
\begin{eqnarray*}
&&
\left|\X_3+\frac{1}{\X_1}-1\right|^2=4\frac{|\X_3|}{|\X_1|}\cos(\A_2)\cos(\A_4),\\
&&
\left|\overline{\X_3}+\frac{1}{\X_1}-1\right|^2=4\frac{|\X_3|}{|\X_1|}\cos(\A_1)\cos(\A_3),
\end{eqnarray*}
\begin{eqnarray*}
&&
\left|\frac{1}{\X_2}+\frac{1}{\X_3}-1\right|^2=\frac{4}{|\X_2||\X_3|}\cos(\A_3)\cos(\A_4),\\
&&
\left|\frac{1}{\X_2}+\frac{1}{\overline{\X_3}}-1\right|^2=\frac{4}{|\X_2||\X_3|}\cos(\A_1)\cos(\A_2).
\end{eqnarray*}

\end{lem}

\begin{proof}
The proof of the first identity is incorporated in the proof of Lemma \ref{lem:giT}. All the other equalities are proved in an analogous manner.
%We only prove the first identity; the proof of the rest is similar: one has to rearrange Equations  (\ref{eq:cross1}) and (\ref{eq:cross2}) to take the other two equivalent pairs of defining equations of $\fX$.

%We  have
%\begin{eqnarray*}
%|\X_1+\X_2-1|^2&=&|\X_1|^2+|\X_2|^2-2\Re(\X_1+\X_2)+1+2\Re(\X_1\overline{\X_2})\\
%\text{from Eq. (\ref{eq:cross2})}\;
%&=&2|\X_1|^2\Re(\X_3)+2\Re(\X_1\overline{\X_2})\\
%\text{from Eqs (\ref{eq:X-A})}\;
%&=&2|\X_1||\X_2|\cos(\A_4-\A_1)+2|\X_1||\X_2|\cos(\A_1+\A_4)\\
%&=&4|\X_1||\X_2|\cos(\A_1)\cos(\A_4),
%\end{eqnarray*}
%where for the second equality we have used Equation (\ref{eq:cross2}) and for the third equality we have used Proposition \ref{prop:X-A}.
\end{proof}

%Before doing that, we observe that rearranging in Eqs. \ref{eq:cross1} and \ref{eq:cross2} we take two equivalent pairs of defining equations of $\fX$, namely,
%\begin{eqnarray}\label{eq:cross1.2}
% &&
%|\X_2|^2=|\X_1|^2|\X_3|^2,\\
%&&\label{eq:cross2.2}
%\frac{2}{|\X_1|^2}\Re(\X_2)=|\X_3|^2+\frac{1}{|\X_1|^2}-2\Re\left(\frac{1}{\X_1}\right)-2\Re(\X_3)+1,
%\end{eqnarray}
%and also  
%\begin{eqnarray}\label{eq:cross1.3}
% &&
%|\X_1|^2=|\X_2|^2/|\X_3|^2,\\
%&&\label{eq:cross2.3}
%\frac{2}{|\X_2|^2}\Re(\X_1)=\frac{1}{|\X_3|^2}+\frac{1}{|\X_2|^2}-2\Re\left(\frac{1}{\X_2}\right)-2\Re\left(\frac{1}{\X_3}\right)+1.
%\end{eqnarray}

\subsection{Properties and Structures of Singular Sets}\label{sec:ssagain}

\subsubsection{Real and ${\rm CR}$ singular sets}
Lemma \ref{lem:X-A} induces %First, we observe that the description of $\fX_\R$ and $\fX_{\C\R}$ although symmetric is somewhat redundant. This is shown in the next 
two corollaries from which we obtain the manifold description of the real singular set $\fX_\R$ and the ${\rm CR}$ singular sets $\fX_{\rm{CR}}$ and $\fX_{\rm{CR}}^{*}$. 

\begin{cor}\label{cor:sing}
 Let  $\fp=(p_1,p_2,p_3,p_4)$ be a quadruple of pairwise distinct points in $\partial\bH^2_\C$ and let $\X_i(\fp)$, $i=1,2,3$ be its cross--ratios. The following hold: 
\begin{enumerate}
\item  [{i)}] $\X_1(\fp)+\X_2(\fp)=1$ if and only if either $p_1,p_2,p_3$ or  $p_2,p_3,p_4$ lie in the same $\C-$circle. In the first case
$$
\X_3(\fp)+\frac{1}{\X_1(\fp)}=1\quad\text{and}\quad \frac{1}{\X_2(\fp)}+\frac{1}{\X_3(\fp)}=1,
$$
and in the second case
$$
\overline{\X_3(\fp)}+\frac{1}{\X_(\fp)}=1\quad\text{and}\quad \frac{1}{\X_2(\fp)}+\frac{1}{\overline{\X_3(\fp)}}=1.
$$
\item [{ii)}] $\X_3(\fp)+\frac{1}{\X_1(\fp)}=1$ if and only if either $p_1,p_3,p_4$ or  $p_1,p_2,p_3$ lie in the same $\C-$circle. In the first case
$$
\X_1(\fp)+\X_2(\fp)=1\quad\text{and}\quad\frac{1}{\X_2(\fp)}+\frac{1}{\X_3(\fp)}=1,
$$
and in the second case
$$
\X_1(\fp)+\overline{\X_2(\fp)}=1\quad\text{and}\quad \frac{1}{\X_2(\fp)}+\frac{1}{\overline{\X_3(\fp)}}=1.
$$
\item [{iii)}] $\frac{1}{\X_2(\fp)}+\frac{1}{\X_3(\fp)}=1$ if and only if either $p_1,p_2,p_3$ or  $p_1,p_2,p_4$ lie in the same $\C-$circle. In the first case
$$
\X_1(\fp)+\X_2(\fp)=1\quad\text{and}\quad\X_3(\fp)+\frac{1}{\X_1(\fp)}=1,
$$
and in the second case
$$
\X_1(\fp)+\overline{\X_2(\fp)}=1\quad\text{and}\quad\overline{\X_3(\fp)}+\frac{1}{\X_2(\fp)}=1.
$$
\end{enumerate}
\end{cor}

In the specific case where all points of $\fp$ lie in a $\C-$circle we have:

\begin{cor}\label{cor:XR-conditions}
 Let  $\fp=(p_1,p_2,p_3,p_4)$ be a quadruple of points in $\partial\bH^2_\C$. The following are equivalent:
\begin{enumerate}
 \item [{i)}] All $p_i$ lie in the same $\C-$circle;
\item [{ii)}] $\X_i(\fp)\in\R$, $i=1,2$ and $\X_1(\fp)+\X_2(\fp)=1$;
\item [{iii)}] $\X_i(\fp)\in\R$, $i=1,3$ and $\X_3(\fp)+\frac{1}{\X_1(\fp)}=1$;
\item [{iv)}] $\X_i(\fp)\in\R$, $i=2,3$ and $\frac{1}{\X_2(\fp)}+\frac{1}{\X_3(\fp)}=1$;
\item [{v)}] $\X_i(\fp)\in\R$, $i=1,2,3$ and $\X_3=-\X_2/\X_1$. 
\end{enumerate}

\end{cor} 

 We note that condition i) $\Leftrightarrow$ v) is Proposition  5.13 of \cite{PP}.  The next proposition describes the differentiable structure of $\fX_\R$, $\fX_{\rm{CR}}$ and  $\fX_{\rm{CR}}^*$. %Note that $\fX_\R$ is actually a set of small dimension.

%\medskip

\begin{prop}\label{prop:XR-XCR}
Consider the singular sets $\fX_\R$, $\fX_{\rm{CR}}$ and  $\fX_{\rm{CR}}^*$ of cross--ratio variety $\fX$. The following hold:
\begin{enumerate}
 \item [{i)}] The real singular set $\fX_\R$  is a 1--dimensional disconnected real manifold, isomorphic to the real line $x+y=1$ with the points $(0,1)$ and $(1,0)$ removed.
 \item [{ii)}] The  ${\rm CR}$ singular set $\fX_{\rm{CR}}$ and the singular set $\fX_{\rm{CR}}^*$ are 1--dimensional complex manifolds, both biholomorphic to $\C-\{0,1\}$. Moreover, the real singular set $\fX_\R$ is contained in $\fX_{\rm{CR}}$ (and in  $\fX_{\rm{CR}}^*$) as a disconnected real submanifold.
\end{enumerate}
\end{prop}

\subsubsection{The complex singular set}\label{sec:csset}

Finally, we turn our attention to the complex singular set $\fX_\C$. First, we make the following observation: From the defining Equation (\ref{eq:cross2}) of $\fX$ and the obvious inequality $-\Re(\X_3)\le\X_3\le\Re(\X_3)$ we have
\begin{equation}\label{eq:ineq-X}
 \left(|\X_1|-|\X_2|\right)^2\le 2\Re(\X_1+\X_2)-1\le\left(|\X_1|+|\X_2|\right)^2.
\end{equation}
If $\Im(\X_3)=0$, then $\Re(\X_3)=|\X_3|$ or $\Re(\X_3)=-|\X_3|$ and thus we have either
\begin{equation}\label{eq:Xseq}
\left(|\X_1|-|\X_2|\right)^2=2\Re(\X_1+\X_2)-1\quad\text{or}\quad 2\Re(\X_1+\X_2)-1=\left(|\X_1|+|\X_2|\right)^2,
\end{equation}
respectively. Therefore, excluding all points of $\fX$ at which $\Im(\X_3)=0$, we obtain strict inequalities in (\ref{eq:ineq-X}).

\begin{prop}\label{prop:Xs}
The following holds.
$$
\fX_\C=\{(\X_1,\X_2,\X_3)\in\fX\;|\;\X_3>0\}\cup\fX_\R.
$$
Moreover, the set of triples of cross--ratios corresponding to quadruples that lie in an $\R-$circle or in a $\C-$circle is contained in $\fX_\C$. 
\end{prop}

\begin{proof}
We only prove the second assertion of our proposition.
% If for a point $(\X_1,\X_2,\X_3)\in\fX$ we have $\Im(\X_3)=0$, then $\X_3\in\R$ and   it is easy to derive \ref{eq:Xseq} from the equations of $\fX$.
All triples of cross--ratios corresponding to quadruples that lie in an $\R-$circle are contained in $\fX_\C$. This is because in that case (cf. Proposition 5.14 of \cite{PP}) $\X_i>0$ for all $i=1,2,3$ and moreover
$$
\left (\X_1-\X_2\right)^2=2\X_1+\X_2-1.
$$
Next, consider  triples of cross--ratios corresponding to quadruples that lie in a $\C-$circle. For such a triple we have %(cf. Proposition 5.13 of \cite{PP})  
that $\X_i\in\R$ and $\X_1+\X_2=1$. The following possibilities can occur:
\begin{enumerate}
\item $\X_1\X_2>0,\; \X_3=-\X_2/\X_1<0$ and
\item $\X_1\X_2<0,\; \X_3=-\X_2/\X_1>0$. 
\end{enumerate}
In case (1) the right Equation (\ref{eq:Xseq}) is satisfied; the left Equation (\ref{eq:Xseq}) is satisfied in case (2).
\end{proof}

We may now  write the complex singular set $\fX_\C$ of the cross--ratio variety $\fX$ as the disjoint union of $\fX_\C^1$ and $\fX_\C^2$ where
\begin{eqnarray}
&&\label{eq:Xs1}
\fX_\C^1=\{(\X_1,\X_2,\X_3)\in\fX\:|\;\left (|\X_1|-|\X_2|\right)^2=2\Re(\X_1)+2\Re(\X_2)-1\},\\
&&\label{eq:Xs2}
\fX_\C^2=\{(\X_1,\X_2,\X_3)\in\fX\:|\;\left (|\X_1|+|\X_2|\right)^2=2\Re(\X_1)+2\Re(\X_2)-1\}.
\end{eqnarray}

\begin{prop}
Let $\fX_\C$ be the complex singular set of $\fX$ and let also $\fX_\C=\fX_\C^1\cup\fX_\C^2$ where $\fX_\C^1$ and $\fX_\C^2$ are as in (\ref{eq:Xs1}) and (\ref{eq:Xs2}), respectively.
\begin{enumerate}
\item $\fX_\C^*=\fX_\C^1-\{(\X_1,\X_2,\X_3)\;|\; \X_1+\X_2=1,\;\X_1\X_2>0\}$ admits the structure of a 3--dimensional submanifold of $\C^2$.
\item $\fX_\C^2$ is a 1--dimensional  disconnected manifold diffeomorphic to the disjoint union of the two open open line segments given by $x_1+x_2=1$, $x_1x_2<0$. 
\end{enumerate}
\end{prop}

\begin{proof}
We only sketch the proof of (1). Let $\zeta_i=x_i+iy_i$, $i=1,2$ and the equation
$$
F(\zeta_1,\zeta_2)=\left(|\zeta_1|-|\zeta_2|\right)^2-2\Re(\zeta_1+\zeta_2)+1=0.
$$ 
We have
\begin{eqnarray*}
&&
\frac{\partial F}{\partial x_1}=2\left[\left(1-\frac{|\zeta_2|}{|\zeta_1|}\right)x_1-1\right],\quad
\frac{\partial F}{\partial x_2}=2\left[\left(1-\frac{|\zeta_1|}{|\zeta_2|}\right)x_2-1\right],\\
&&
\frac{\partial F}{\partial y_1}=2\left(1-\frac{|\zeta_2|}{|\zeta_1|}\right)y_1,\quad
\frac{\partial F}{\partial y_2}=2\left(1-\frac{|\zeta_1|}{|\zeta_2|}\right)y_2,
\end{eqnarray*}
and the reader may verify that all partial derivatives vanish at points where $y_1=y_2=0$ and $x_1+x_2=1$, $x_1x_2>0$.
\end{proof}

A codimension 1 $\rm{CR}$ structure  is called strictly pseudoconvex if the Levi form is strictly positive.  We have the following:

\begin{prop}\label{prop:levipseudo}
There is a strictly pseudoconvex $\rm{CR}$ structure of codimension 1 defined on $\fX_\C^*$. 
\end{prop}

\begin{proof}

Consider the equation
$$
F(\zeta_1,\zeta_2)=\left(|\zeta_1|-|\zeta_2|\right)^2-2\Re(\zeta_1+\zeta_2)+1=0.
$$
We calculate 
\begin{eqnarray*}
&&
\frac{\partial F}{\partial \overline{\zeta_1}}=\left(1-\frac{|\zeta_2|}{|\zeta_1|}\right)\overline{\zeta_1}-1,\quad
\frac{\partial F}{\partial \overline{\zeta_2}}=\left(1-\frac{|\zeta_1|}{|\zeta_2|}\right)\overline{\zeta_2}-1,\\
&&
\frac{\partial^2 F}{\partial \zeta_1\partial\overline{\zeta_1}}=1-\frac{|\zeta_2|}{2|\zeta_1|},\quad 
\frac{\partial^2 F}{\partial \zeta_1\partial\overline{\zeta_2}}=-\frac{\overline{\zeta_1}\zeta_2}{2|\zeta_1||\zeta_2|},\quad
\frac{\partial^2 F}{\partial \zeta_2\partial\overline{\zeta_1}}=-\frac{\overline{\zeta_2}\zeta_1}{2|\zeta_1||\zeta_2|},\quad
\frac{\partial^2 F}{\partial \zeta_2\partial\overline{\zeta_2}}=1-\frac{|\zeta_1|}{2|\zeta_2|}.
\end{eqnarray*}
Observe   that all partial derivatives of the first order vanish at points $(\zeta_1,\zeta_2)$ such that $\zeta_1,\zeta_2\in\R$ and $\zeta_1+\zeta_2=1$.
Now, straightforward calculations show that at points $(\zeta_1,\zeta_2)$ such that $F(\zeta_1,\zeta_2)=0$ we have:
\begin{equation*}
L(\zeta_1,\zeta_2)=1+\frac{\Re(\zeta_1\overline{\zeta_2})}{|\zeta_1||\zeta_2|}.
\end{equation*}
The above is in general greater or equal than zero; in the case where $L(\zeta_1,\zeta_2)=0$ we have $\Re(\zeta_1\overline{\zeta_2})=-|\zeta_1||\zeta_2|$ and then
\begin{eqnarray*}
0&=&\left(|\zeta_1|-|\zeta_2|\right)^2-2\Re(\zeta_1+\zeta_2)+1\\
&=&|\zeta_1|^2+|\zeta_2|^2+2\Re(\zeta_1\overline{\zeta_2})-2\Re(\zeta_1+\zeta_2)+1\\
&=&|\zeta_1+\zeta_2-1|^2
\end{eqnarray*}
which is not the case here. The proof is complete.
\end{proof}

From Proposition \ref{prop:levipseudo} it follows that the set $\p$ as in Section \ref{sec:complexI} is Levi strictly pseudoconvex with smooth boundary $\fX^*_\C$, see p.128 of \cite{Kr} for the definition of Levi pseudoconvexity.

\end{document}